\documentclass[USletter, 11pt]{article}
\usepackage{amsmath, amsfonts, graphicx, dsfont, xcolor}
\usepackage{enumitem}
\usepackage[english]{babel}
\usepackage[left=1in, right=1in, top=1in]{geometry}

\newcommand{\QED}{\hfill\rule{0.5em}{0.5em}\\}
\newcommand{\modu}[0]{\operatorname{{mod}}}

\newenvironment{proof}[1][Proof]%
{\begin{description}\item[\noindent\textbf{#1}:]}
{\QED\end{description}}
\newtheorem{theorem}{Theorem}[section]

\newtheorem{lemma}[theorem]{Lemma}

\newtheorem{corollary}[theorem]{Corollary}

\newtheorem{note}[theorem]{Note}
\newtheorem{definition}{Definition}

\author{Drago Bokal\\Faculty of Natural Sciences and Mathematics,\\
University of Maribor, Slovenia \\
  \texttt{drago.bokal@um.si}
  \and
Gunnar Brinkmann\\ Applied Mathematics, Computer Science and Statistics,\\
Ghent University,\\
Krijgslaan 281 S9,\\
B 9000 Ghent, Belgium\\
 \texttt{gunnar.brinkmann@ugent.be}
 \and
Carol T. Zamfirescu\footnote{also: Department of Mathematics, Babe\c{s}-Bolyai University, Cluj-Napoca, Roumania} \\ Applied Mathematics, Computer Science and Statistics,\\
Ghent University,\\
Krijgslaan 281 S9,\\
B 9000 Ghent, Belgium\\
 \texttt{czamfirescu@gmail.com}
}

\title{The Connectivity of the Dual}

\begin{document}

\date{}

\maketitle

\bigskip

\begin{abstract}

The dual of a polyhedron is a polyhedron --
or in graph theoretical terms: the dual of a $3$-connected plane graph is
a $3$-connected plane graph.
Astonishingly, except for sufficiently large facewidth,
not much is known about the connectivity of the dual on higher surfaces.
Are the duals of
$3$-connected embedded graphs of higher genus $3$-connected, too? If not, which
connectivity guarantees $3$-connectedness of the dual? In this article, we give answers to
some of these and related questions. We prove that there is no connectivity
that guarantees the $3$-connectedness or $2$-connectedness of the dual  for every genus,
and give upper bounds for the minimum genus for which (with $c>2$) a $c$-connected
embedded graph with a dual that has a $1$- or $2$-cut can occur.
 We prove that already on the torus,
we need $6$-connectedness
to guarantee $3$-connectedness of the dual and $4$-connectedness
to guarantee $2$-connectedness of the dual.

In the last section, we answer a related question
by Plummer and Zha on orientable embeddings of highly connected non-complete graphs.

\end{abstract}

%\begin{keyword}graph \sep embedding \sep genus \sep  dual \sep connectivity
%\MSC 5C10 Planar graphs, geometric and topological aspects; 5C40 Connectivity
%\end{keyword}

\maketitle

\newcommand{\ignore}[1]{}

\section{Introduction}

Relations between dual polyhedra have been observed at least since Kepler in 1619 \cite{kpl},
but it was not until several centuries later that duality was formally defined. One of the first definitions
was given by Br\"{u}ckner \cite{bru,wen}. With graph embeddings, duality can be abstractly defined
for any embedded graph: the dual $G^\ast$ of an embedded graph $G$ is an embedded graph, whose vertices
are faces of $G$; two faces of $G$ being adjacent as vertices of $G^\ast$ whenever they share an
edge in $G$ \cite{moh,graphs_on_surfaces}.
It is folklore that the dual of a plane polyhedral embedding (i.e.\ of a 3-connected plane graph)
is again a plane polyhedral embedding and in the planar case a graph and its dual are often considered to be
{\em almost the same} -- e.g. some algorithms for listing certain classes of cubic plane graphs
work in fact by generating the corresponding dual graphs -- that is: triangulations -- and dualizing them for output.
Mohar (\cite{moh}, Proposition 3.8, Proposition 3.9, Proposition 3.2)
generalized the relation between the connectivity of an embedded graph and its dual to higher surfaces in the following restricted setting
 -- with $\operatorname{fw}(G)$ the facewidth of the
embedded graph $G$:

\begin{theorem}[\cite{moh}]
\label{pr:base}
Let $G$ be an embedded graph of genus $g>0$,
$G^\ast$ the dual embedded graph,
and $c\in \{1,2,3\}$.
Then, the following are equivalent:
\begin{itemize}
\item $\operatorname{fw}(G)\ge c$ and $G$ is $c$-connected,
\item $\operatorname{fw}(G^\ast)\ge c$ and $G^\ast$ is $c$-connected.
\end{itemize}
\end{theorem}

The theorem cannot be extended to $c\ge 4$ due to triangular faces, which -- except for trivial cases -- imply a 3-cut in the dual.

In this contribution, we are mainly interested in {\em simple graphs}, that is graphs that have neither double edges nor loops, but {\em multigraphs},
that is graphs with multiedges and loops allowed,
are used as tools too. In places where it is not obvious from the context which kind of graph is dealt with, we explicitly use the term {\em simple graph}
or {\em multigraph} instead of just {\em graph}.
We study simple embedded graphs with simple duals
and the general case of the
relationship between the connectivity of the graph and its dual
without restrictions on the facewidth. All cuts discussed in this article are vertex cuts.

Our main results are that already on the torus, even the dual of a $5$-connected
graph need not be $3$-connected and that, for each $c>0$, there is a genus $g$
and an embedded $c$-connected graph $G$ of this genus so that the dual has a $1$-cut.
We give
upper bounds for the minimum genus $g$ with this property.
Note that in simple graphs there is an essential difference between $3$-cuts
on one side and $1$- and $2$-cuts on the other: unlike $3$-cuts,
$1$- or $2$-cuts  in the dual cannot occur as trivial cuts resulting from facial cycles.

Unless explicitly mentioned otherwise, all embedded graphs in this article
are connected. They come with a combinatorial embedding
in an oriented manifold. We will deal with embeddings only in their combinatorial representation. The equivalence of
this representation with the topological description is well described in standard books on topological graph theory
like \cite{top_graph_theory,graphs_on_surfaces}.
A combinatorial embedding in an oriented manifold is given in the following way: all undirected edges $\{v,w\}$ are interpreted
as a pair of directed edges $(v,w)$ and $(w,v)$, where $(v,w)^{-1}=(w,v)$. For each vertex $v$,  all incident
directed edges $(v,.)$ are assigned
a cyclic order around the vertex, called the \emph{(local) rotation}, so that for an edge $(v,x)$, we can talk about the previous and next
edge around $v$. The set of all cyclic orders is called a \emph{rotation system}.
Faces are cyclic sequences $(v_0,v_1),(v_1,v_2), \dots ,(v_{k-1},v_0)$ of pairwise different
directed edges, so that, for $0\le i \le k-1$, $(v_{(i+1)\modu k},v_{(i+2)\modu k})$ follows the edge
$(v_{(i+1)\modu k},v_i)$ in the cyclic order around $v_{(i+1)\modu k}$.  In this case, we call
$(v_{(i+1)\modu k},v_i),(v_{(i+1)\modu k},v_{(i+2)\modu k})$ an \emph{angle} of the face and say that the face has
size $k$. The genus $g(G)$ of an embedded graph $G$ is given by the Euler formula
$v(G)-e(G)+f(G)=2-2g(G)$ with $v(G)$, $e(G)$, and $f(G)$ the number of vertices, edges and faces, respectively.
This must not be mixed up with the genus of an abstract, not embedded, graph, which is defined as the minimum
of all genera of the different embeddings of the graph. In this article we will only discuss genera of embedded graphs.
The Euler formula and the fact that in a simple graph each face contains at least three edges
imply that $\left\lceil \frac{(c-2)(c-3)}{12}\right\rceil$ is a lower bound on the genus on which any simple graph with minimum degree $c$
can be embedded. The fact that this bound is best possible is the celebrated map colour theorem  \cite{genus_complete_graph} determining
the genus of complete graphs.

We will not only investigate whether some connectivity guarantees 3-connectedness of the dual, but also whether other connectivities
can be guaranteed and in how far this depends on the genus of the embedded graph. To this end we define the function $\delta_k(c)$:

For $c\ge k \ge 1$ we define $\delta_k(c)$ as the minimum genus $s$ so that there is a simple $c$-connected embedded graph $G$
with genus $s$, so that the dual graph $G^*$ is a simple graph with a $k$-cut.

At this point it is not yet clear that such a minimum genus exists, but it will turn out that  $\delta_k(c)$ is well defined for all $c\ge k \ge 1$.

%\[
%\begin{array}{ll}
%\delta_k(c)=  \min \{  s  |  & \exists \text{ simple }c\text{-connected embedded graph } G \text{ with genus } s \text{, so that}\\
% & \text{the dual graph } G^* \text{ is a simple graph with a } k-\text{cut }\}.
%\end{array}
%\]

Some values for $\delta_k(c)$ are known or can be easily determined.
To determine some of the others, we need some definitions and basic results.

\section{Notation and basic results}

In this article, a face and the corresponding vertex in the dual graph are denoted
by the same symbol, so that it makes sense to write $v\in f$ for a vertex $v$ of a graph
and a vertex $f$ of the dual graph if it is contained in a directed edge of the face $f$.

\begin{itemize}

\item Let $G=(V,E)$ be a simple embedded graph and $V_c\subset V$ a cutset in $G$.
  A {\em boundary face} is a face $(v_0,v_1),(v_1,v_2), \dots ,(v_{k-1},v_0)$,
  so that there exist $0\le i<j\le k-1$ with $v_i\in V_c$ and $v_j\in V_c$.
  Note that
  $v_i=v_j$ is possible.
The set
  $F_b$ is the set of all boundary faces.
  For a component $C$ of
$G-V_c$, let $F_b(C)$ be the subset of faces of $F_b$ that contain
  at least one vertex of $C$.

\item The embedded {\em boundary multigraph} $G_b$ is the bipartite
graph with vertex set
$V_c\cup F_b$, where a vertex $v\in V_c$ is adjacent to $f\in F_b$ if $v\in f$.
For each time $v$ occurs in
the closed boundary walk of $f\in F_b$,  there is an edge $\{v,f\}$ and the embedding is given by the rotation
 around $v$, respectively by the boundary walk. We consider the embedding to be
given by the rotation system, so the genus of $G_b$ is bounded from above by the genus of $G$. In general, $G_b$ needs not be connected,
but in all cases where we apply the Euler formula to $G_b$, it will be connected.
Face sizes in this graph will later be used to determine bounds on the size of cut sets in the dual.

\item The multigraph $\bar G_b$ is the graph $G\cup G_b$ where the rotation  around
vertices in $V_c$ is such that the edges to vertices in $F_b$ are inside the corresponding faces.
The genus of $\bar G_b$ is equal to the genus of $G$.

\item For a component $C$ of $G-V_c$, the set $F_C^{int}$ is the set of ({\em interior}) faces  of $C$, that is faces
of $G$ that are not in $F_b$
and contain only vertices of $C\cup V_c$ in the boundary.

\item  If $G$ is an embedded multigraph and $G'=(V',E')$ a subgraph with an embedding induced by $G$,
a ($G'$-)bridge $B$ of $G$ is either a subgraph of $G$ that is a single edge of $G-E'$ with
both ends in $V'$, or a component $C_B$
of $G-V'$ together with the edges of $G$
with one endpoint in $V'$ and one in $C_B$
and the endpoints of these edges
in $G'$.
We say that a bridge $B$ is {\em inside a face} $f$ of $G'$,
if there is an angle $e_1,e_2$ of $f$
so that there is an edge of $B$ in the rotation between $e_1$ and $e_2$.

\item Bridges can be inside different faces. If for a face of $G'$ we have that all bridges
inside this face are inside no other face of $G'$, we call this face {\em simple}, otherwise
{\em bridged}.
Note that if a face $f$ is bridged,
there is at least one other bridged face $f'$
with which $f$ shares bridges, but there could be more.

\item Let $G$ be an embedded multigraph, $G'$ a subgraph with an embedding induced by $G$
and $f$ a simple face of $G'$. We define the {\em internal component} of $f$ as follows: we first replace each vertex
$v$ that occurs $k>1$ times in the facial walk around $f$ by
pairwise different vertices $v_1,\dots ,v_k$. If the angle at the $i$-th occurence of
$v$ is $(v,x),(v,y)$, the neighbours of $v_i$ and the rotation
are given by all edges $(v,z)$ in the cyclic order around $v$ from
$(v,x)$ to $(v,y)$ -- including $(v,x)$ and $(v,y)$.
The internal component of $f$ is then given by all vertices and edges on the modified boundary walk
(which is now a simple cycle)
of $f$ together with all bridges inside $f$. This implies that the boundary corresponding to $f$
in the internal component of $f$ is always a simple cycle.

If the internal component of a simple face $f$ has genus $0$,
we call this face a {\em simple internally plane} face, otherwise a {\em simple internally non-plane} face.

\end{itemize}

The following lemma has the combinatorial version of the Jordan curve
theorem as the special case $g=g'=0$ and $G'$ a cycle. If a subgraph $G'$ of a graph $G$ has bridged faces, then
the edges in a bridge connecting two faces force a higher genus of $G$ than that of $G'$. The same is true for simple faces
containing a bridge in the interior that has itself already a nonzero genus. This is formalized in the following lemma.

\begin{lemma}\label{lem:jordan}

Let $G$ be an embedded multigraph of genus $g$
and $G'$ an embedded (also connected) subgraph of genus $g'$ with the embedding induced by $G$.
Let $b$ denote the number of bridged faces of $G'$ and $s_{np}$ denote the number of
simple internally non-plane faces.
Then $s_{np}+\frac{b}{2} \le g-g'$.

\end{lemma}

\begin{proof}

Note first that if $f$ is a simple face of $G'$ such that the internal
component $C$ has genus $g_{C}$, the subgraph $G'_f$ of $G$ that consists of
all vertices and edges of $G'$ and $C$ has genus $g'_f=g'+g_C$: if $v',e',f'$, respectively $v_C,e_C,f_C$
and $v'_f,e'_f,f'_f$ are the numbers of vertices, edges
and faces of $G'$, resp.\ $C$ and $G'_f$, then -- with $l$ the length of the boundary cycle of $f$
in $C$  -- we have

\[v'_f - e'_f + f'_f = (v'+v_C-l)-(e'+e_C-l)+(f'+f_C-2)=\]
\[ (v'-e'+f')+(v_C-e_C+f_C)-2. \]

This gives by Euler's formula

\[g'_f=\frac{2-(v'_f - e'_f + f'_f)}{2}=\frac{2-((v'-e'+f')+(v_C-e_C+f_C)-2)}{2}=\]
\[\frac{2-((2-2g')+(2-2g_C)-2)}{2}=\frac{2g'+2g_C}{2}=g'+g_C.\]

We will prove the result by induction on the number $d$ of edges that are in $G$, but not in $G'$.
If $d=0$, we have $G=G'$ and $s_{np}=b= g-g'=0$, so the result holds.

If $d>0$ and there is a simple internally plane face $f$ of $G'$ with an internal
component $C$, that contains a bridge,
we can apply induction to $G'_f$, and as (with the notation from above) $g_C=0$,
neither $s_{np}$ nor $b$ or $g'$ change,
the result follows immediately.

If $d>0$ and $s_{np}>0$, let $f$ be a simple internally non-plane face of $G'$ and
$G'_f$ as above. For $G'_f$, we have (with the notation from above)
that $g'_f=g'+g_C\ge g'+1$, and with $s'_{np}$ respectively $b'$ the number of
simple internally non-plane faces of $G'_f$,
respectively the number of bridged faces of $G'_f$,
we have $b'=b$ and $s'_{np}=s_{np}-1$. By induction $s'_{np}+\frac{b'}{2} \le g-g'_f$,
so $s_{np}+\frac{b}{2} = (s'_{np}+1)+\frac{b'}{2} \le g- g'_f +1\le g-g'$.

Let now $f\not= f'$ be bridged faces of $G'$ so that there is a bridge $B$
inside both faces $f$ and $f'$. Let $e,e'$ be edges of $B$ with endpoints in
$f$, $f'$, respectively. Note that $e=e'$ is possible if $B$ is a single edge.
In $B$, there is a path  starting in an
angle of $f$ with edge $e$ and ending in an angle
of $f'$ with edge $e'$. Adding this path to $G'$ to obtain $G'_P$, we
get a graph with the same faces as $G'$ -- except for $f,f'$, which
become one new face.

Because in addition the number of edges added is one larger than the number of vertices added,
we have, for the genus $g'_P$ of $G'_P$,
that $g'_P=g'+1$. Old simple internally non-plane faces
are not changed, but the new face can be
a new simple internally non-plane face. So if $s'_{np}$ denotes the new number
of simple internally non-plane faces, we have $s'_{np}\ge s_{np}$.
The new face can be simple or bridged, but in any case, with $b'$ the number of
bridged faces of $G'_P$, we have $b'\ge b-2$ because other bridged faces stay bridged.
We get
\[ s_{np}+\frac{b}{2} \le s'_{np} + \frac{b'+2}{2} \le s'_{np} + \frac{b'}{2} +1 \le
  g- g'_P +1 =  g-g'.\]

\end{proof}

Let $V_c$ be a cutset of a simple embedded graph $G$ and $f$ be a face. If $f$ contains vertices $v_1,v_2$ of different components
of $G-V_c$, then going from $v_1$ in the two possible directions along $f$ we reach at least two different positions in the facial walk
with vertices from $V_c$ before reaching $v_2$. This implies that $f\in F_b$.

\begin{lemma}\label{lem:boundarycut}

Let $G=(V,E)$ be a simple embedded graph and $V_c\subset V$ a cutset, so that,
for at least two components $C_1,C_2$ of $G-V_c$, we have that $F_{C_1}^{int}$ and $F_{C_2}^{int}$
are not empty.
Then  $F_b(C_1)$ is a cutset in the dual graph $G^*$.

\end{lemma}

\begin{proof}

Let $f_1\in F_{C_1}^{int}, f_2\in F_{C_2}^{int}$ and $v_1\in f_1, v_1\not\in V_c$ and $v_2\in f_2, v_2\not\in V_c$.
If $G^*-F_b(C_1)$ is connected, then there is a path $f_1=f'_1,f'_2,f'_3,\dots ,f'_n=f_2$ in
$G^*-F_b(C_1)$. Let $f'_i$ be the first face that is not in $F_{C_1}^{int}$. Since it is adjacent to
$f'_{i-1}$ in $G^*-F_b(C_1)$, it shares an edge with $f'_{i-1}$ in $G$, so it shares at least
one vertex from $C_1$ with $f'_{i-1}$.
This means that $f'_i\in F_{C_1}^{int}$ or $f'_i\in F_b(C_1)$
-- both of which are impossible.

\end{proof}

Note that such a cutset $F_b(C_1)$ in $G_b$ can contain vertices from different faces
of $G_b$ if the component is bridging two or more faces.

\begin{corollary}
\label{cr:dualcut}
Let $G=(V,E)$ be a simple embedded graph and $V_c\subset V$ a cutset, so that
$G_b$ as subgraph of
$\bar G_b$ has a simple face $\bar f_1$
whose interior and exterior contain faces of $G$.
Then  $\{f\in F_b| f\in \bar f_1\}$ (note that a face $f\in F_b$ is also a vertex in
$G_b$) is a cutset in the dual graph $G^*$ of size at most $\frac{l}{2}$ if
$l$ is the number of directed edges in $\bar f_1$.

\end{corollary}

\begin{lemma}\label{lem:atleast5}

  Let $G$ be a simple embedded graph with a $1$-cut $\{v_c\}$ and a simple dual.
  If $G_b$ has a vertex $f_0\in F_b$ with  (as face of $G$) an internal component $C$ and $F_C^{int}=\emptyset$, then
  $|F_b|\ge 5$. Note that, for a $1$-cut $\{v_c\}$, $G_b$ is connected.

\end{lemma}

\begin{proof}

  Let $C$ be such a component and $x_1$ a neighbour of $v_c=x_0$ in
  $C$.  Let $(x_0,x_1),$ $(x_1,x_2),\dots,$ $(x_{i-1},x_i)$ be a maximal
  path in the face $f_0\in F_b$ containing $(x_0,x_1)$, so that
  $ v_c \not\in \{x_1,\dots,x_{i-1}\}$. We have $x_i=v_c$, but the path cannot be the whole face, because in that case, the face would contain $v_c$
  only once and would therefore be in $F_C^{int}$. As $G$ has no double
  edges, we have $i\ge 3$.  For $0<j\le i$, we denote the face
  containing $(x_j,x_{j-1})$ by $f_j$, so the faces $f_0,f_1,\dots
  ,f_i$ are pairwise different, since the dual has no double edges or
  loops.  See the left hand side of Figure~\ref{fig:paths} for an
  illustration.

\begin{figure}[h!t]
	\centering
	\includegraphics[width=0.6\textwidth]{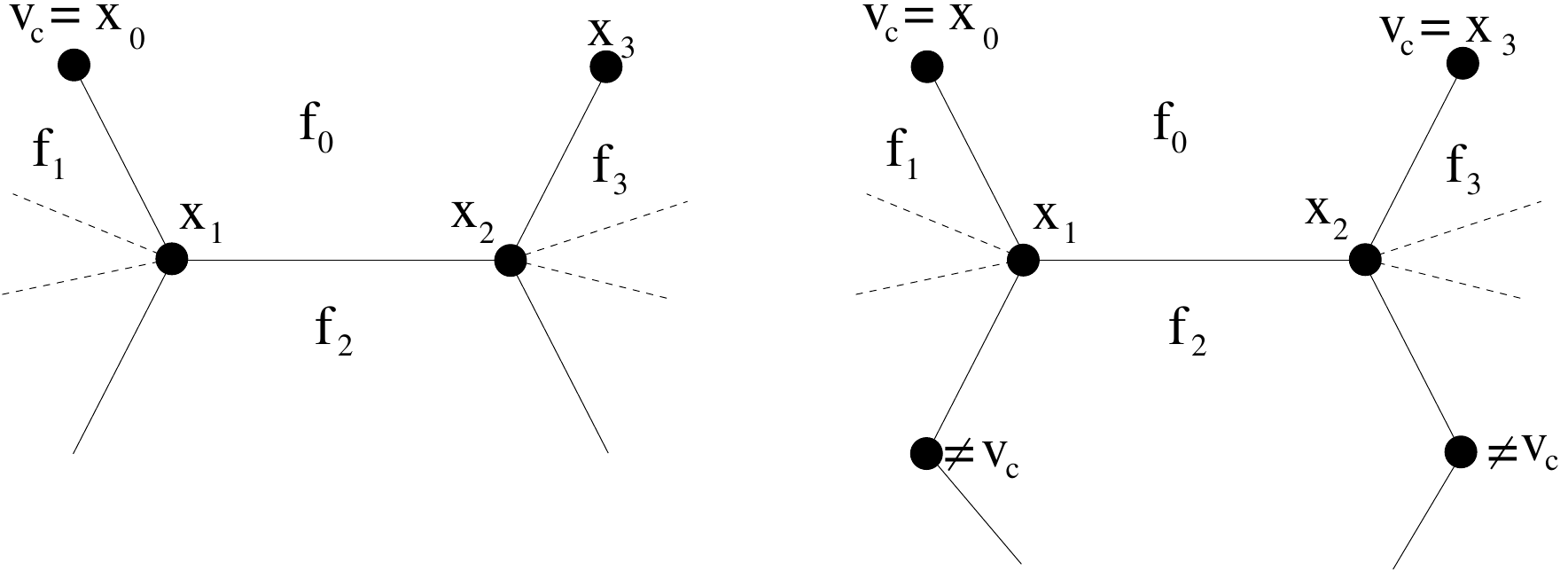}
	\caption{Paths in faces of $F_b$.}
	\label{fig:paths}
\end{figure}

Because all faces $f_0,f_1,\dots ,f_i$ contain vertices of $C$,
they are either in $F_C^{int}$ or in $F_b$,
and because $F_C^{int}=\emptyset $, they are in $F_b$. If $i\ge4$,
we have that $|F_b|\ge 5$, so assume
that $i=3$. Then (see the right part of Figure~\ref{fig:paths}) $(x_2,x_1)$ is contained in
a facial path starting and ending with $v_c$ showing that,
in that case, even $|F_b|\ge 6$.

\end{proof}

\begin{lemma}\label{lem:2face}

  Let $G$ be a simple embedded graph with a simple dual and a given cutset $V_c$.
  If $G_b$ has a face $f$ of size $2$, then (as subgraph of $\bar G_b$) there is exactly one component $C$ inside $f$
  and (no matter whether $f$ is bridged or not) $F_C^{int}\not= \emptyset$.

\end{lemma}

\begin{proof}

  Let $f_b,v_c$ be the vertices forming the $2$-face $f$ in $G_b$ and
  let the order around $v_c$ (in $\bar G_b$) inside $f$
be $(v_c,f_b),(v_c,x_1),\dots ,(v_c,x_k),(v_c,f_b)$.
Then $k\ge 2$ because otherwise the face $f_b$ in $G$ would imply a loop in the dual.
Furthermore,
  for $1\le i <k$, the edges $(x_i,v_c)$ and $(v_c,x_{i+1})$
  belong to the same face $f_i$. This face  $f_i$
  contains $v_c$ as the only element of $V_c$ and only at one position, because otherwise we would
  not have a $2$-face in $G_b$. This means that $f_i\in F_C^{int}$ and that
  the face boundary without $v_c$ connects $x_i$ and $x_{i+1}$
  so that they
  belong to the same component. Since each component inside $f$ must contain at least one of
  these vertices, there is only one component.

\end{proof}

\begin{lemma}\label{lem:classes}

  Let $G$ be a simple embedded graph with a simple dual and a given cutset $V_c$ of $G$.
  Each cycle $Z$ in a component $C$ of $G-V_c$ is either
  \begin{description}
  \item[(a)] in a simple face $f$ of $G_b$ so that a face of the internal component is also a face of $G$, or
  \item[(b)] in a bridged face of $G_b$, or
  \item[(c)] in a simple internally non-plane face of $G_b$.

  \end{description}

\end{lemma}

\begin{proof}

If $C$ is connected to more than one face, we have (b), so assume that
$C$ is inside a simple face $f$. Embedding $f$ (with vertices occuring more than
once replaced by copies), $Z$ and a path from $f$ to $Z$, we have a plane graph
with (a directed version of) $Z$ forming a face $f_0$ inside $f$. If $f_0$ is a
face of $G$, we are in case (a). Otherwise we can recursively
argue, that if we add the remaining edges to form the internal component containing $C$ in the cyclic
order given by $G$, we either connect two faces in one of the steps (which means
that we have case (c)) or we will just subdivide $f_0$ producing faces inside of it,
which means that we end up in case (a).

\end{proof}

\begin{corollary}\label{cor:classes}

  Let $G$ be a simple embedded graph with a simple dual and a $1$-cut $\{v_c\}$.
  Then each $G_b$-bridge $B$ of $\bar G_b$ is
  \begin{description}
  \item[(a)] in a simple face of $G_b$ containing an interior face of $G$, or
  \item[(b)] in a bridged face, or
  \item[(c)] in a simple face with nonplane interior.
  \end{description}

\end{corollary}

\begin{proof}

  Because $\{v_c\}$ is a $1$-cut, there are no bridges of $G_b$ that
are just edges, since they would be loops. Furthermore the minimum degree in $G$ is $3$, because otherwise the dual would
have a double edge, which implies that each vertex has at least two neighbours that are in the same component of
$G-\{v_c\}$. Together this implies that each bridge of $G_b$ contains a cycle, so that the result follows with
Lemma~\ref{lem:classes}.

\end{proof}

\begin{lemma}\label{lem:in different faces}

Let $G$ be a simple embedded graph with a given cutset $V_c$.
If $C_1,C_2$ are different $G_b$-bridges of $\bar G_b$, then they are
in different faces of $G_b$.

\end{lemma}

\begin{proof}

Let $f$ be a face of $G_b$. If $v_c, v'_c\in V_c$ are two vertices
following each other in the cyclic order around $f$ with one
vertex $f_b$ representing a face of $G$ in between, then there
is a path in $f_b$ and inside $f$ connecting
$v_c$ and $v'_c$ without vertices of $V_c$ in between. This path is part of a bridge,
so each two vertices of $V_c$ following each other in the cyclic order around $f$
are contained in a common bridge.
If we can show
that there are no two different bridges inside $f$ sharing a
vertex at the same angle in $f$,
this implies that there are no two such different bridges
at vertices following each other in the cyclic order and finally
that there is only one bridge.

Now let $v_c\in V_c$ be a vertex of a face $f$ of $G_b$ and $C_1,C_2$
be two different bridges sharing $v_c$ at the same position of the face $f$.
Without loss of generality assume that
if $(f_{b,1},v_c),(v_c,f_{b,2})$ are the edges of $f$ at that position,
that the first edge
in the rotation  from $(v_c,f_{b,1})$ to $(v_c,f_{b,2})$ around
$v_c$ belongs to $C_1$ and that the first edge $(v_c,c_2)$ that does not belong
to $C_1$ belongs to $C_2$. Let $(v_c,c_1)$ be the previous edge of $(v_c,c_2)$, so $c_1\in C_1$.
As there is no edge to a face in $F_b$ between $(v_c,c_1)$ and $(v_c,c_2)$
in $\bar G_b$, a vertex of $V_c$ occurs only once in
the face containing $(c_1,v_c), (v_c,c_2)$, so the path connecting $c_1$ and
$c_2$ along the side not containing $v_c$ shows that they belong to the same bridge
-- a contradiction.

\end{proof}

The following lemma will be used to show the existence of small faces in the boundary multigraph $G_b$
of a graph $G$ with a given 1-cut. Together with Lemma~\ref{lem:boundarycut} these small faces will
then imply a small cut in the dual.

\begin{lemma}\label{lem:smallfaces}

 Let $G=(V_1\cup V_2,E)$ be an embedded bipartite multigraph of genus $g$ with bipartition classes
$V_1,V_2$, so that $|V_1|=1$ and that each vertex in $V_2$ has degree at least $2$.
If, for some $k$,
we have that $i$ faces have size less than $2k$, then

\begin{center}
  $i\ge \frac{k}{k-1} -\frac{2k}{k-1}g + \frac{k-2}{k-1}|V_2|.$
 \end{center}

If additionally there is at most one face of size $2$, then

\begin{center}
$i\ge \frac{k-1}{k-2} -\frac{2k}{k-2}g + |V_2|.$
\end{center}

\end{lemma}

\begin{proof}

Summing up the face sizes and using lower bounds
for the faces of size at least $2k$ as well as for the
smaller faces, we get with $f$ the number of faces and $e$ the number of edges of $G$

$2e\ge 2k(f-i) +2i, \qquad $ so $ \qquad f\le \frac{e+(k-1)i}{k}.$

Inserting this, with $1+ |V_2|$ as the number of vertices of $G$ into the Euler formula we get

$2-2g=1+ |V_2|-e+f \le 1+ |V_2|-e+ \frac{e+(k-1)i}{k}= 1+ |V_2|-\frac{k-1}{k}e+\frac{k-1}{k}i.$

Since $e\ge 2|V_2|$ we get

$2-2g \le 1+ |V_2|-\frac{2k-2}{k}|V_2|+\frac{k-1}{k}i =  1 - \frac{k-2}{k}|V_2|+\frac{k-1}{k}i, \quad$ so

$\frac{k-1}{k}i  \ge 1 -2g + \frac{k-2}{k}|V_2|, \qquad $ thus $ \qquad  i\ge \frac{k}{k-1} -\frac{2k}{k-1}g + \frac{k-2}{k-1}|V_2|$

which is the first result.

If there is at most one face of size $2$, then

$2e\ge 2k(f-i) +4i -2,  \qquad $ so $ \qquad f\le \frac{e+(k-2)i +1}{k}.$

Starting with this formula, a completely analogous computation gives the second result.

\end{proof}

\section{Results on small genus}

\begin{lemma}\label{lem:existcut}
~\\
\begin{description}
\item[(a)] Let $G$ be a simple embedded graph with a $1$-cut
 that has a simple dual $G^*$.
  \begin{itemize}
  \item If $g(G)=1$, then $G^*$ has a cut of size at most $3$.
  \item If $g(G)=2$, then $G^*$ has a cut of size at most $5$.
    \end{itemize}

\item[(b)] Let $G$ be a simple embedded graph with a $2$-cut
  that has a simple dual $G^*$.

  If $g(G)=1$, then $G^*$ is at most $5$-connected.

\end{description}
\end{lemma}

\begin{proof}

(a) Let $v$  be a cutvertex of $G$.
Due to Corollary~\ref{cr:dualcut}, it is sufficient to show
that $G_b$ has a simple face $\bar f_1$ with a face
of $G$ in the interior and exterior
and, for $g=1$, boundary length at most $6$
(and thus $|F_b(C)|\le 3$ for some component $C$)
 and, for $g=2$,
boundary length at most $10$ (and thus $|F_b(C)|\le 5$).

Due to Lemma~\ref{lem:in different faces}, $G_b$ has at least two faces. If
$|F_b|< 5$, they both contain a face of $G$ (Lemma~\ref{lem:atleast5}), so that
$F_b$ is a cutset in the dual (Lemma~\ref{lem:boundarycut}).
If $g=1$ and $|F_b|< 4$ or $g=2$ and $|F_b|< 5$ we are done.

We will use that the graph $G_b$ satisfies the conditions of
Lemma~\ref{lem:smallfaces} with $V_1=\{v\}$ and $V_2=F_b$.

Now assume $g=1$ and $|F_b(C)|\ge 4$. Then -- with $g'$ the genus of
$G_b$, $i$ as in Lemma~\ref{lem:smallfaces}, and $k=4$ -- Lemma~\ref{lem:smallfaces} gives:

$$i\ge \frac{4}{3} -\frac{8}{3}g' + \frac{8}{3}=4-\frac{8}{3}g'.$$

If $g'=1$, then $i\ge \frac{4}{3}$, so  $i\ge 2$. Due to Corollary~\ref{cor:classes} and
Lemma~\ref{lem:jordan}, at least two faces of $G_b$ with at most $3$ elements of $F_b$
in the boundary are simple and contain an interior face of $G$, so that we can apply Lemma~\ref{lem:boundarycut}.

If $g'=0$, then $i\ge 4$, so again at least two faces of $G_b$
with at most $3$ elements of $F_b$
in the boundary are simple and internally planar and contain an interior face of $G$.
Again Lemma~\ref{lem:boundarycut} gives the desired result.

\medskip

Assume now $g=2$ and $|F_b(C)|\ge 5$. Then -- with $g'$ the genus of
$G_b$ and $k=6$ -- Lemma~\ref{lem:smallfaces} gives:

$$i\ge \frac{6}{5} -\frac{12}{5}g' + 4.$$

If $g'=2$, then $i\ge \frac{2}{5}$,
so there is a face $f$ of $G_b$ with at most $5$ elements of $F_b$
in the boundary and, due to Corollary~\ref{cor:classes} and Lemma~\ref{lem:jordan},
this and also
at least one other face are simple, internally planar,
and have an interior face of $G$.
We can apply Lemma~\ref{lem:boundarycut} to $f$.

If $g'=1$, then $i\ge \frac{14}{5}$, so $i\ge 3$.
This implies that there is at least one simple, internally planar face of $G_b$
with an interior face of $G$ and at most $5$ elements of $F_b$ in the boundary.
If there is another simple
internally planar face of $G_b$ or another face of $G_b$ with an internal face of $G$, we are done,
but in principle it is possible that there are just $3$ faces of $G_b$, and that
the other two are bridged and do not contain an internal face of $G$.
In this case, we have
(Lemma~\ref{lem:2face}) that there is at most one face of $G_b$ of size $2$ and the second part of
Lemma~\ref{lem:smallfaces} gives:

$$i\ge \frac{5}{4} -\frac{12}{4} + 5=\frac{13}{4}.$$

To this end $i\ge 4$ and there is at least one more simple
internally planar face of $G_b$ -- which in fact even has a short boundary.

If $g'=0$, then $i\ge \frac{26}{5}$, so $i\ge 6$ and it follows immediately that we have
at least two simple internally planar faces of $G_b$ with sufficiently short boundary.

\bigskip

(b) Suppose that $G^*$ has no cut of size at most $5$, but is the simple dual
of a graph $G$ with a $2$-cut.
The simple 6-connected toroidal graphs have been described in \cite{hexagontorus}. They are the duals
of hexagonal tilings of the torus and can be parametrized by three values $p,r > 0$, $0\le q < p$.
The construction is given in Figure~\ref{fig:hexagontorus}: for a segment of $p\times r$ hexagons
of the hexagonal lattice, the upper and lower
as well as the left and right boundaries are identified and
the left ($p$-) part is shifted by $q$ positions before identification.
For small values of $p,r$, the
graph or the dual can have multiple edges, but if the graph and the dual are simple, the graph
$G$ -- that is, the hexagonal tiling -- is also 3-connected:

Let $v_1,v_2$ be two vertices, $V'=V\setminus \{v_1,v_2\}$ and $h$
an arbitrary hexagon. If $h$ does not contain a vertex of $\{v_1,v_2\}$, all vertices of
$V'\cap h$ belong to the same component of $V'$. Assume $v_1\in h$.
Because the dual is simple, all vertices in the three hexagons around $v_1$
are pairwise distinct. This implies that all vertices in the boundary cycle of the three
hexagons that are also in $V'$ belong to the same component (even if the cycle contains $v_2$).
So for each hexagon, all vertices in the boundary
that are not in $V'$ belong to the same component.
Since the dual is $6$-connected, for each pair of hexagons, a path of hexagons can be found
showing that the boundary vertices are in the same component -- so $G$ has no $2$-cut.

\begin{figure}[h!t]
	\centering
	\includegraphics[width=0.6\textwidth]{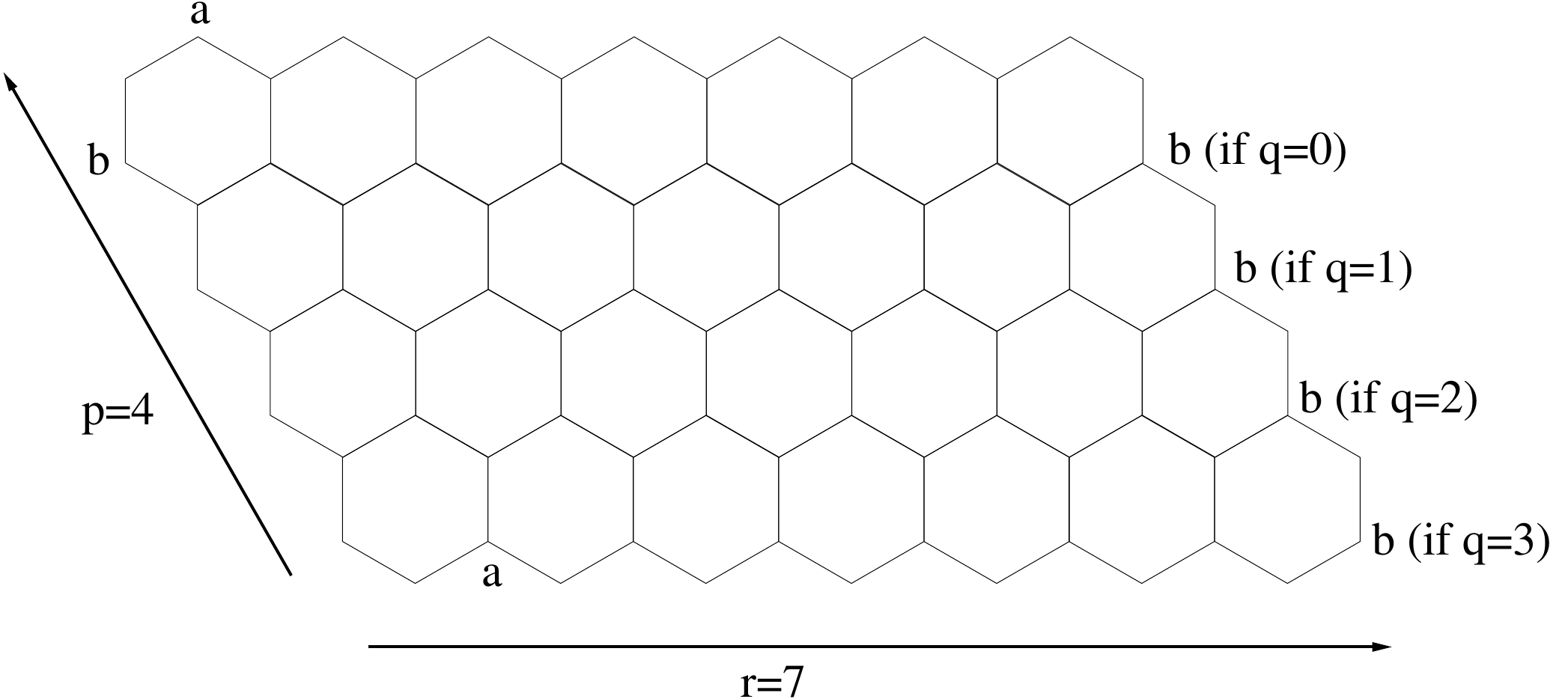}
	\caption{The parametrization of hexagonal tilings on the torus.}
	\label{fig:hexagontorus}
\end{figure}

\end{proof}

Without the assumption of $G^*$ being simple, the statement (a) of Lemma~\ref{lem:existcut} is not true.
An easy counterexample would be the dual of $K_7$ embedded on the torus with
an extra vertex of degree 1 added inside one of the faces. The dual of this graph with a 1-cut
would be $6$-connected: $K_7$ with an extra loop at one of the vertices.

\section{The H-operation}

In order to prove an upper bound for $\delta_2(c)$ we will now describe an operation that introduces a 2-cut in an embedded graph
without changing the (abstract) dual graph.

\begin{definition}

Let $G$ be a simple embedded graph and $x,x'$ be different  vertices of $G$ with
 $e_1,e_2,\dots ,e_n$ the rotation of incident edges around $x$
and $e'_1,e'_2,\dots ,e'_m$  the rotation around $x'$.
Then we say that the graph
where the vertices $x,x'$ are replaced by one vertex $y$
with rotation $e_1,e_2,\dots ,e_n,e'_1,e'_2,\dots,$ $e'_m$
of incident edges
is obtained from $G$ by \emph{identifying the angles $e_n,e_1$ and $e'_m,e'_1$}.

\end{definition}

By counting vertices, edges, and faces it is easy to see that if two angles in different faces
are identified, the genus is increased by one, and if two angles in the same face are identified,
the genus remains the same.

\begin{definition}

Let $G$ be a simple embedded graph with a simple dual and minimum degree at least $2$.
Let $x,y$ be adjacent vertices of $G$ with
degree $3$ and pairwise different neighbours.
Let the rotations around $x$ respectively $y$ be (in vertex notation)
$y,w',v$ respectively $x,w,v'$ (compare Figure~\ref{fig:operation}), $a_1$ be the
vertex before $x$ in the rotation around $v$, $b_m$ be the vertex after $x$
in the rotation  around $w'$, $a_n$ be the vertex after $y$
in the rotation  around $w$, and $b_1$ be the vertex before
$y$ in the rotation  around $v'$.

Then the result of identifying the two angles $(v,a_1),(v,x)$ and $(v',b_1),(v',y)$
and also the angles $(w,y),(w,a_n)$ and $(w',x),(w',b_m)$ is called
the result of the \emph{H-operation} applied to the edge $\{x,y\}$.
We write $H_{\{x,y\}}(G)$. See Figure~\ref{fig:operation}
for an illustration. It is possible that the H-operation produces double edges
and loops.

\end{definition}

\begin{figure}[h!t]
	\centering
	\includegraphics[width=0.8\textwidth]{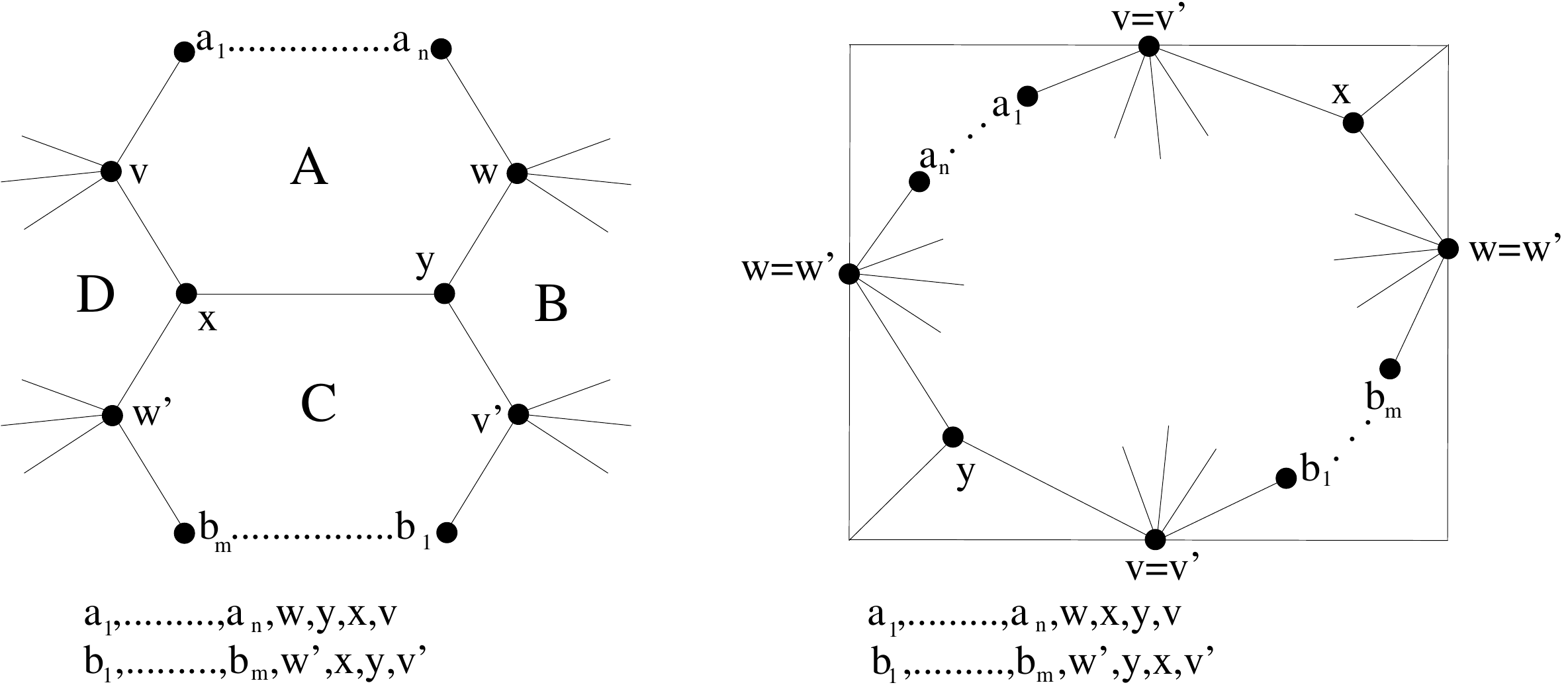}
	\caption{The H-operation applied to an embedded graph and the vertices in the facial walk around the
          boundaries of the old and new faces.}
	\label{fig:operation}
\end{figure}

After one of the angle identifications, the genus is increased by one, but the second identification
is then applied to angles in the same face, so the H-operation increases the genus only by one. After
the operation, the former vertices $v,w,v',w'$ are identified to $2$ vertices that separate $x$ and $y$
from the rest. The H-operation has an impact on two faces that are replaced by two other faces.
Following the face boundaries of the new faces, one sees that there is a 1-1 correspondence
between the old and new faces that induces an isomorphism of the dual graph.
In fact it would not have been necessary to require a simple dual, so that the faces $A,C$ in Figure~\ref{fig:operation}
are different, but as we only need the operation in this restricted setting, we only discussed the case of a simple dual.
We will condense
these observations in a note:

\begin{note}\label{note:h-operation}

  If $G$ is a simple embedded graph of genus $g$ with a simple dual and an edge $e$ to
  which the H-operation can be applied without producing double edges
  or loops, then $H_e(G)$ is a simple graph of genus $g+1$ with a 2-cut and a
  dual graph that is isomorphic to the dual of $G$.

\end{note}

\begin{lemma}\label{lem:h-possible}

Let $G$ be a simple embedded graph with all faces of size at least $5$ that has a simple dual.
If $v_1$ is a vertex where all vertices at distance at most $2$ of $v_1$ have degree $3$,
then, for each edge $e$ incident with $v_1$, the graph $H_{e}(G)$ is simple.

\end{lemma}

\begin{proof}

We use the notation of Figure~\ref{fig:h-possible}
and without loss of generality let $e=\{v_1,v_2\}$. We have to show that $v_4$ and $v_6$ are different, non-adjacent
and do not have a common neighbour,
i.e.\  that the distance $d(v_4,v_6)$ is at
least $3$. Furthermore we have to show that the same is true for $v_3$ and $v_5$ -- even
after $v_4$ and $v_6$ have been identified.

\begin{figure}[h!t]
	\centering
	\includegraphics[width=0.35\textwidth]{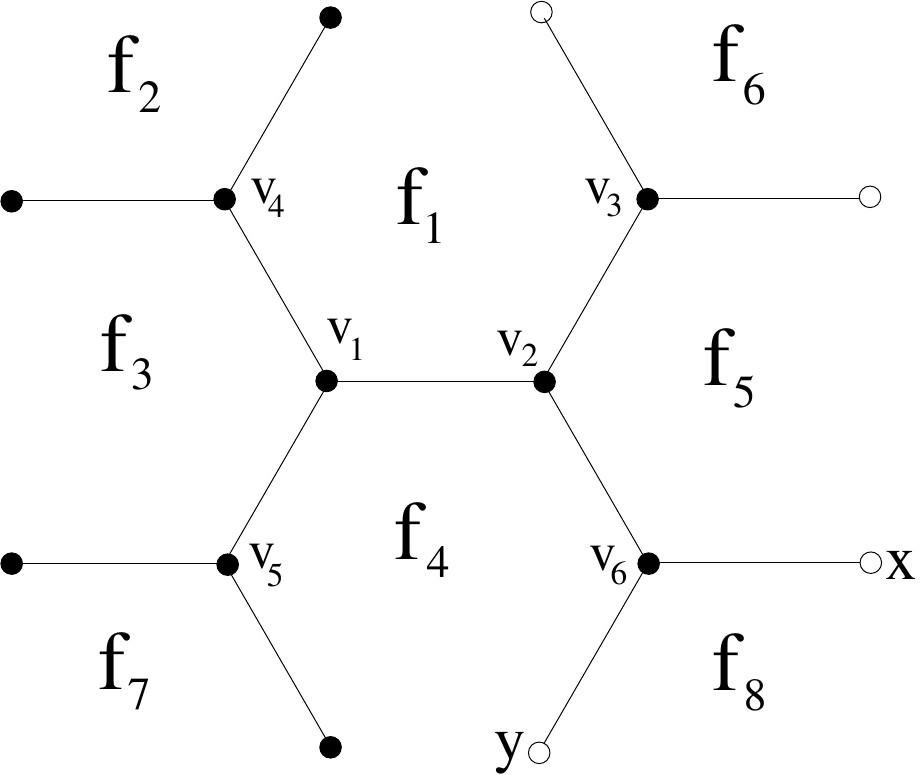}
	\caption{A part of an embedded graph with a vertex $v_1$, so that all vertices at distance at most $2$
have degree $3$. Vertices with degree $3$ are marked by a filled circle, while vertices with degree $3$ or larger
are marked by an empty circle.}
	\label{fig:h-possible}
\end{figure}

As the dual is simple, the faces $f_1,f_2,\dots ,f_6$ are pairwise distinct.
If we had, e.g., $v_4=v_6$, then $f_5\in \{f_1,f_2,f_3\}$,
which would imply a loop or double edge at $f_1$ in the dual.
Analogously, we can conclude that
$v_1,\dots ,v_6$ are pairwise distinct.

Furthermore, $v_3,v_4,v_5,v_6$ are pairwise non-adjacent: \\
$v_4$ cannot be adjacent to $v_5$ since this would imply $f_2=f_3$ or that
$f_3$ is a triangle. The same argument shows that there is no edge $\{v_3,v_6\}$,
because both have degree $3$.
In addition,
$v_4$ cannot be adjacent to any other vertex $v_i\not= v_1$ in the boundary of $f_4$
(and analogously for $v_5$ and $f_1$).
At this neighbour (which would have degree $3$) at least one of the faces $f_1,f_3$ would have a double edge
with $f_4$ in the dual.

If $\{v_3,v_4\}$ was an edge of $G$, we had $f_6\in \{f_1,f_2,f_3\}$, which is again impossible.
The case for  $\{v_5,v_6\}$ is symmetric.

Suppose now that $d(v_4,v_6)=2$. Then $v_4$ is adjacent to $x$ since $y$
is in the boundary of $f_4$. Then $x$ had degree $3$, so that
$\{f_5,f_8\}\cap \{f_1,f_2,f_3\}\not= \emptyset$.  The only
possibility that does not immediately imply double edges or loops in
the dual is $f_8=f_2$. Looking at the rotation  around $v_4$, we
would get that $f_5$ would share a second edge with $f_1$.

After $v_4$ and $v_6$ have been identified, a shortest path between $v_3$ and $v_5$ was already
present before the identification (and thus have length at least $3$) or contain $v_4=v_6$ after
the identification as an intermediate vertex. If this path had length $2$, $v_3$ and $v_5$ would
be adjacent to $v_4$ or $v_6$ already before the identification, which is not the case,
so $d(v_3,v_5)\ge 3$ (and in fact equal to $3$) after the identification of $v_4$ and $v_6$.

\end{proof}

\begin{lemma}\label{lem:bigfaces}

For $c\ge 3$, let $K_{c+1}$
be embedded with minimal genus
and $d\ge 0$ be minimal so that $(c-2)(c-3)+d \equiv 0\;\;(\modu 12)$.
Let $F_L$ denote the set of faces that are not triangles and $s(f)$ the size of a face $f$.
Then, $K_{c+1}$ has
$f= \frac{c^2}{3} + \frac{c}{3} - \frac{d}{6}$ faces
and we have $\sum_{f'\in F_L} (s(f')-3)=  \frac{d}{2}$.
\end{lemma}

\begin{proof}

The minimum genus $g$ is $\left\lceil \frac{(c-2)(c-3)}{12}\right\rceil$ \cite{genus_complete_graph}.
The number $e$ of edges in
$K_{c+1}$ is $\frac{(c+1)c}{2}$ and the number $v$ of vertices is $c+1$.
So, with $f$ the
number of faces, we get by Euler's formula

%$ 2-2g=v-e+f= c+1 -\frac{(c+1)c}{2} +f$

$ 2- 2\frac{(c-2)(c-3)+d}{12} = c+1 -\frac{(c+1)c}{2} +f$ and thus

%$ 2- \frac{c^2-5c+6+d}{6} = c+1 -\frac{c^2+c}{2} +f$

$f= \frac{c^2}{3} + \frac{c}{3} - \frac{d}{6}$.

For all $f'\in F_L$, we have that $s(f')>3$, so

$2e=3f+\sum_{f'\in F_L} (s(f')-3) \quad$ so

$\quad f=\frac{2e-\sum_{f'\in F_L} (s(f')-3)}{3}= \frac{c^2+c-\sum_{f'\in F_L} (s(f')-3)}{3}$.

Inserting this into the previous equation, we get

$ \frac{c^2+c-\sum_{f'\in F_L} (s(f')-3)}{3}= \frac{c^2}{3} + \frac{c}{3} - \frac{d}{6}$ and finally

$\sum_{f'\in F_L} (s(f')-3)=  \frac{d}{2}$.

\end{proof}

\begin{lemma}\label{lem:completedualsimple}
For $c\ge 3$, the
complete graph $K_{c+1}$
can be embedded in a surface of minimal genus
in a way that the dual is simple.

\end{lemma}

\begin{proof}

In a simple embedded graph with minimum degree at least $2$,
all faces $f$ with $s(f)\le 5$ are simple -- that is:
each vertex $v$ of the face occurs exactly in two directed
edges of the face. Otherwise, the distance between two occurrences
as start- and end-vertex of
a directed edge would be at most $2$, so
the graph would have a loop (distance $1$) or the face would contain a path $(v,v_1),(v_1,v)$
which would imply that there is a double edge or that $v_1$ has degree $1$.

No two different faces in an embedded graph $G$ with minimum degree $3$
and only simple faces
can share two consecutive edges $(v_1,v_2),(v_2,v_3)$ that are part of a face, as in that case
$v_2$ would have degree $2$. This implies that in such a graph a triangle cannot share more than
one edge with another face.

Let now $(v_1,v_2),(v_2,v_3),(v_3,v_4)$ be a subpath in a facial walk of a face $f$
in a (simple) embedded graph $G$ with minimum degree $3$
and simple faces.
We will show that there is no face different from $f$ that is a triangle or quadrangle and
contains two of the directed edges $(v_2,v_1),(v_3,v_2),(v_4,v_3)$.
As shown, no face different from $f$
can contain two of these edges sharing a vertex $v_i$.
The only remaining case is that a quadrangle $f_q$ contains
$(v_2,v_1)$ and $(v_4,v_3)$ and,
in addition for each of $\{v_1,v_3\}$ and $\{v_2,v_4\}$,
exactly one of the two corresponding directed edges.
Because there must be a directed edge with initial vertex $v_1$, $f_q$ must contain $(v_1,v_3)$ -- which implies that
there is no edge in $f_q$ with initial vertex $v_3$. Since each two edges in a quadrangle are contained in a facial path of
length $3$, this implies that two quadrangles cannot share more than one edge with each other.

As a consequence, we have that duals of embeddings of the complete graph with maximum face size $5$
and at most one pentagon are simple, because in a face $f$ that is a quadrangle or a pentagon,
each pair of different edges is contained in a path in $f$ of length $3$.

From the main result of \cite{facedistributions} (Theorem~3.3 in the arXiv paper and Theorem~2.2 in the paper in \emph{Journal of Graph Theory})
it follows that for each $n\ge 3$ there is a minimum genus embedding of $K_n$ with at most one pentagon and all other faces of size at most $4$.
This proves the result.

\end{proof}

\begin{lemma}\label{lem:embedcomplete}

For each $c\ge 6$, the complete graph $K_{c+1}$ can be embedded in a surface of minimal genus $g$ in a way
that the dual is simple and that there is an edge to which the H-operation can be applied without producing double edges or loops.

\end{lemma}

\begin{proof}

Let $K_{c+1}$ be embedded with minimal genus in a way that the dual is simple.
  We want to prove that in the dual, there is a vertex $v$ with only vertices of degree $3$ at distance
  at most $2$ to $v$. Because vertices with degree $3$ in the dual are triangles
  in the primal graph,
  we will discuss triangles in the primal graph.

With the notation of Lemma~\ref{lem:bigfaces},
each $f'\in F_L$ is a vertex in the dual with degree $s(f')$.
We say that a vertex $f$ is {\em blocked}
by a vertex $f'\in F_L$  if  (in the dual) $d(f,f')\le 2$ and $d(f,f')$ is minimal among all vertices $f'\in F_L$.
A vertex $f'$ of degree $s(f')$ can block at most $3s(f')+1$ vertices (including itself).
All vertices in $F_L$ together can block at most

$\sum_{f'\in F_L} (3s(f')+1)= 3\sum_{f'\in F_L} (s(f')-3)+ 10|F_L|=  \frac{3d}{2}+ 10|F_L|$

vertices. As, for $f'\in F_L$, we have $s(f')-3\ge 1$, this implies $|F_L|\le  \frac{d}{2}$ and the number $b_l$ of blocked vertices
in the dual is at most $\frac{13d}{2}$.

If a vertex is not blocked, we can apply the H-operation to any edge incident with it without creating
double edges or loops (Lemma~\ref{lem:h-possible}).

Since $d$ is always even, we have $d\le 10$, so, for $c\ge 14$, we get with Lemma~\ref{lem:bigfaces}

$f= \frac{c^2}{3} + \frac{c}{3} - \frac{d}{6}\ge \frac{410}{6} > \frac{130}{2} \ge  \frac{13d}{2}\ge b_l,$

hence, for $c\ge 14$, there is always a vertex so that we can apply the H-operation to each incident edge.

For $c<14$, consider the following table:

\begin{tabular}{c|c|c|c|c}
$c$ & $d$ & $f$ & upper bound & $b_l$ for just one \\
&    &    &   for $b_l$ & face not a triangle \\
\hline
13 & 10 & 59 & 65 & 25 \\
12 & 6 &   51  & 39 & 19 \\
11 & 0 & 44 & 0 & 0 \\
10 & 4 & 36 & 26 & 16 \\
9  &  6 & 29 & 39 & 19 \\
8 & 6 & 23 & 39 & 19 \\
7 & 4 & 18 & 26 & 16 \\
6 & 0 & 14 & 0 & 0 \\

\end{tabular}

For $c\in \{6,10,11,12\}$, the dual of each embedding has a vertex that is not blocked.
For $c\in \{7,8,9,13\}$, we can only draw this conclusion for an embedding with only one face
that is not a triangle. In the appendix, we give such embeddings with a simple dual for $K_{c+1}$
with $c\in \{8,9,13\}$ to show that they exist. For $c=7$, such an embedding does not exist,
but we give an embedding and an edge to which the H-operation can be applied.

\end{proof}

\section{Bounds and exact values for \boldmath$\delta_k(c)$}

For $k\ge 3$, it is easy to determine the values of $\delta_k()$,
because simple graphs can contain triangles which imply 3-cuts in the dual, but also some other exact values and bounds can now be
determined:

\begin{theorem}\label{lem:somevalues}

\begin{description}
\item[(a)] $\delta_k(c)\ge \left\lceil \frac{(c-2)(c-3)}{12}\right\rceil$ for $c\ge k\ge1$, $c>5$.
\item[(b)] $\delta_1(1)=0$, $\delta_1(2)=\delta_1(3)=1$, $\delta_1(4)=\delta_1(5)=2, \delta_1(6)= 3$.
\item[(c)] $\delta_2(2)=0$, $\delta_2(3)=\delta_2(4)=\delta_2(5)=1$, $\delta_2(6)=2$.  \\
For  $c\ge 7$: $\delta_2(c)\in \{  \left\lceil \frac{(c-2)(c-3)}{12}\right\rceil,    \left\lceil \frac{(c-2)(c-3)}{12}\right\rceil +1 \}$,
\item[(d)] For $k\ge 3$ we have: \\
If $c\in \{3,4,5\}$ and $c\ge k$, then $\delta_k(c)=0$.\\
If $c>5$ and $c\ge k$ then $\delta_k(c)=\left\lceil \frac{(c-2)(c-3)}{12}\right\rceil$.
\end{description}

\end{theorem}

\begin{proof}

(a) The value $\left\lceil \frac{(c-2)(c-3)}{12}\right\rceil$ is on one hand the genus of the complete graph
$K_{c+1}$ \cite{genus_complete_graph}, and on the other the smallest genus on which any graph with minimum degree $c$
can be embedded, so that for $s< \left\lceil \frac{(c-2)(c-3)}{12}\right\rceil$ no $c$-connected
graph embedded in a surface of genus $s$ can exist -- no matter what the structure of the dual is.

(b) The case $\delta_1(1)=0$ is trivial, and for $c\in \{2,3\}$ the well known fact that simple $c$-connected plane
graphs have a simple $c$-connected dual implies $\delta_1(c)\ge 1$. The 3-connected graph embedded in
the torus and
displayed in Figure~\ref{fig:torus_3_1} shows  $\delta_1(2)=\delta_1(3)=1$.
The graph in Figure~\ref{fig:doubletorus_5_1} shows $\delta_1(4)\le 2$, $\delta_1(5)\le 2$,
while equality follows with Lemma~\ref{lem:existcut}, part (a).
The graph in Figure~\ref{fig:tripletorus_6_1}
shows $\delta_1(6)\le 3$ and again equality follows with Lemma~\ref{lem:existcut}, part (a).

(c) The case $\delta_2(2)=0$ is trivial, and for $c\in \{3,4,5\}$ the fact that simple 3-connected plane
graphs have a simple 3-connected dual implies $\delta_2(c)\ge 1$.
Applying Note~\ref{note:h-operation} to the dodecahedron embedded in the plane
and with dual the 5-connected icosahedron implies --
together with Lemma~\ref{lem:h-possible} -- $\delta_2(3)=\delta_2(4)=\delta_2(5)=1$.

Lemma~\ref{lem:existcut}, part (b) implies that $\delta_2(6)>1$, and applying
Note~\ref{note:h-operation} and  Lemma~\ref{lem:h-possible}
to an edge of the Heawood graph embedded in the torus (with dual the
6-connected graph $K_7$), we get $\delta_2(6)=2$.

Already in part (a) we showed that  $\delta_2(c)\ge \left\lceil \frac{(c-2)(c-3)}{12}\right\rceil$.
Applying Note~\ref{note:h-operation} to a suitable edge of the dual of
$K_{c+1}$ embedded in a surface of genus $\left\lceil \frac{(c-2)(c-3)}{12}\right\rceil$, and applying
Lemma~\ref{lem:embedcomplete}, shows
that $\delta_2(c)\le \left\lceil \frac{(c-2)(c-3)}{12}\right\rceil+1$.

(d) For $c\in \{3,4,5\}$ the icosahedron shows that $\delta_k(c)=0$.
For $c>5$ the result follows directly from Lemma~\ref{lem:completedualsimple},
since for $c>5$ embeddings of $K_{c+1}$ on a surface
of genus $\left\lceil \frac{(c-2)(c-3)}{12}\right\rceil$ exist that have a simple dual and triangles.

\end{proof}

The remaining -- and most interesting -- values are $\delta_1(c)$ for $c>6$ and the exact
values of $\delta_2(c)$ for $c>6$. We will not be able to decide which of the two possible
values for  $\delta_2(c)$ is the correct one, but we will be able to achieve some progress on
the problem for $\delta_1(c)$ by giving an upper bound on $\delta_1(c)$.
As the definition of $\delta_k(c)$ requires $G$ to be $c$-connected, it follows directly from the Euler formula that
$\delta_1(c)\in \Omega(c^2)$, so at least in the Omega-notation, the bound we will prove
will be optimal.

\begin{figure}[h!t]
	\centering
	\includegraphics[width=0.4\textwidth]{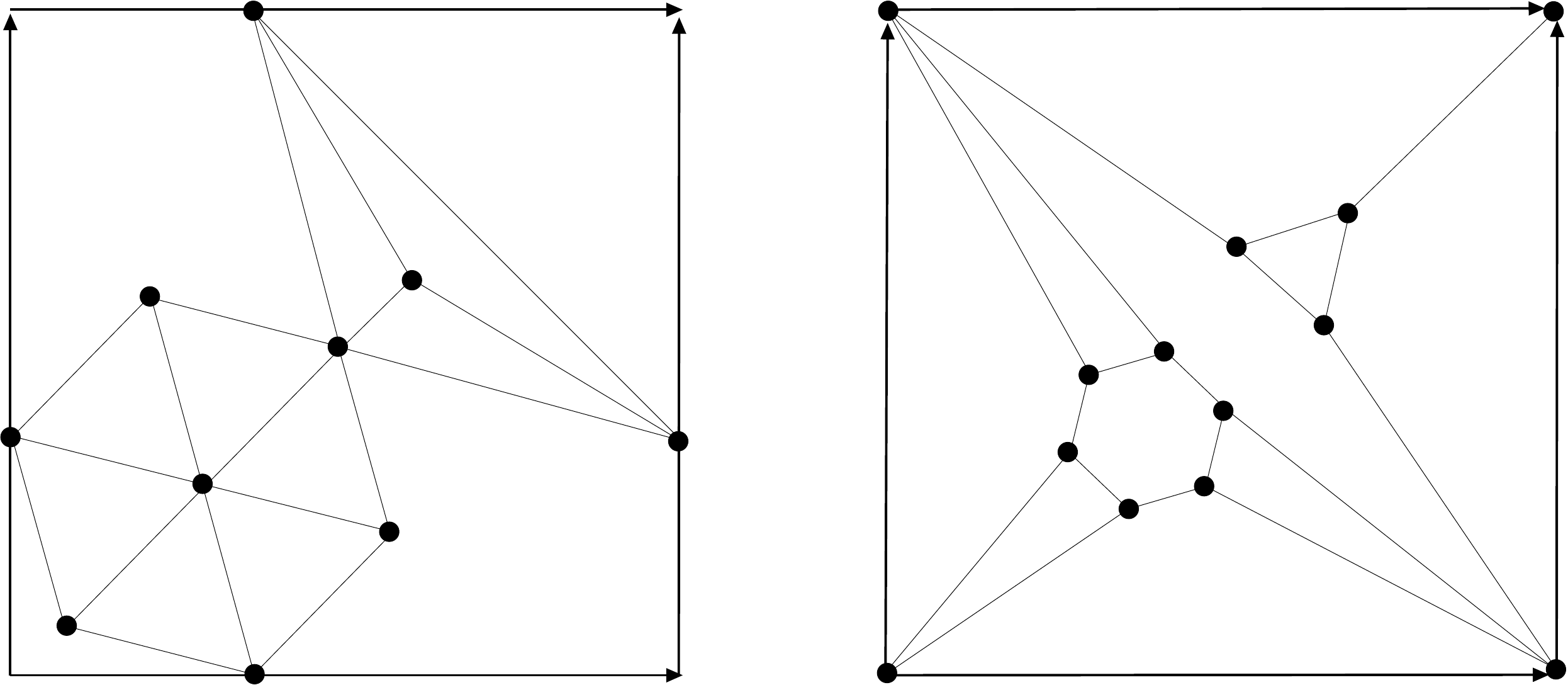}
	\caption{A 3-connected graph on the torus (left) with a dual that has a 1-cut (right).}
	\label{fig:torus_3_1}
\end{figure}

% \begin{figure}[h!t]
% 	\centering
% 	\includegraphics[width=0.4\textwidth]{torus_4_2_new.pdf}
% 	\caption{A 4-connected graph on the torus (right) with a dual that has a 2-cut (left).}
% 	\label{fig:torus_4_2}
% \end{figure}

\begin{figure}[h!t]
	\centering
	\includegraphics[width=0.6\textwidth]{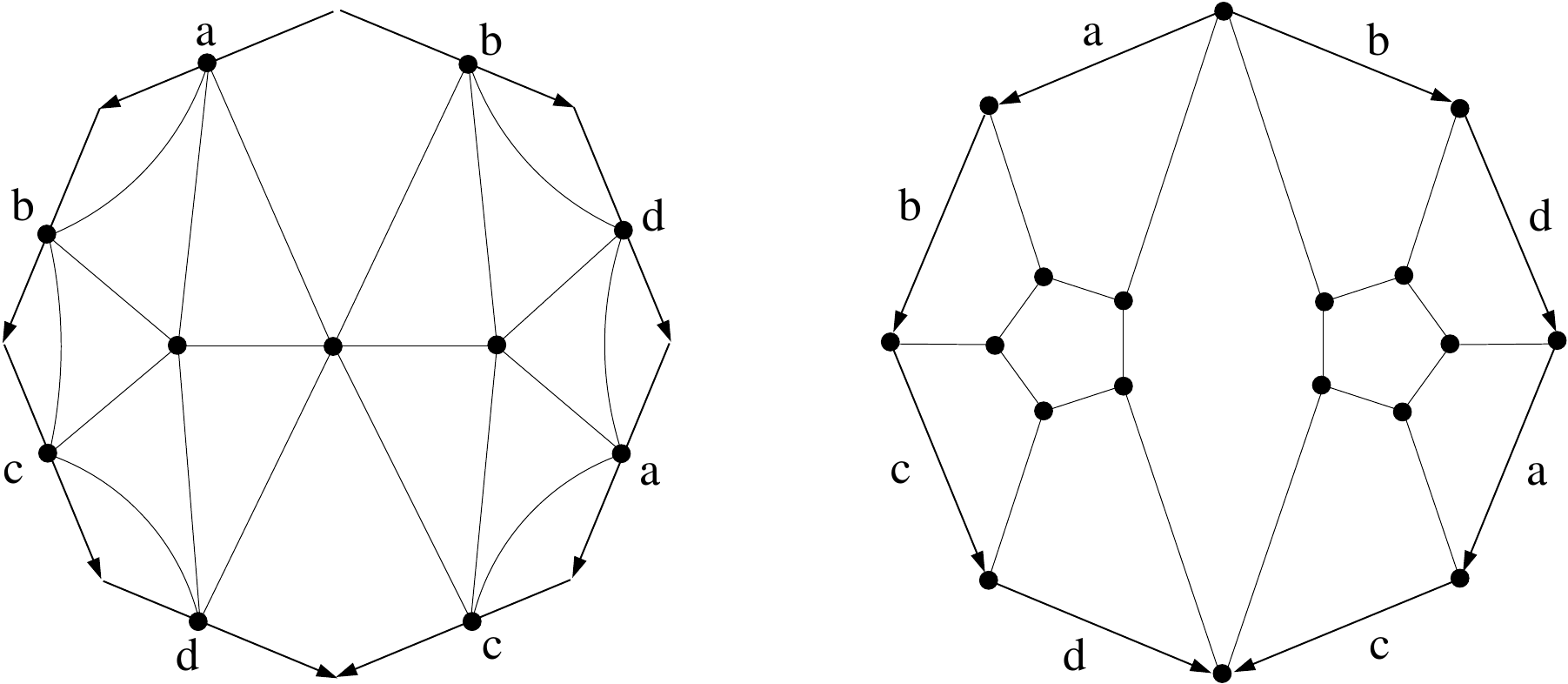}
	\caption{The 5-connected graph $K_7$ minus an edge embedded in the double torus (left) so that the dual has a 1-cut (right).}
	\label{fig:doubletorus_5_1}
\end{figure}

\begin{figure}[h!t]
	\centering
	\includegraphics[width=0.6\textwidth]{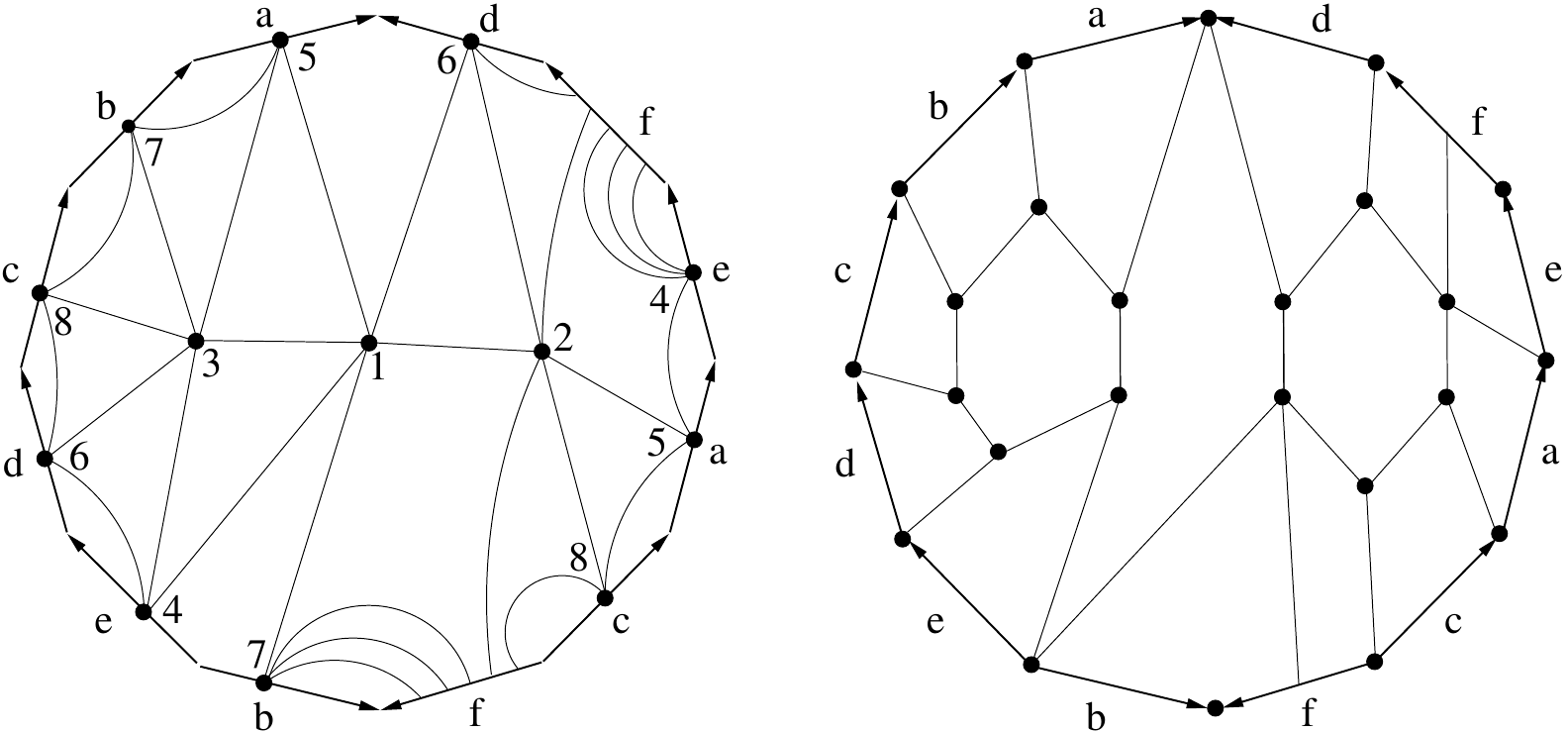}
	\caption{The 6-connected graph $K_8$ minus a matching with three edges embedded in the triple torus (left)
so that the dual (right) has a 1-cut.}
	\label{fig:tripletorus_6_1}
\end{figure}

\begin{theorem}

For $c\ge 7$, we have
$\delta_1(c)\le  \frac{c^2+6c-5}{4}$.

\end{theorem}
\begin{proof}

Let $c\ge 7$, $p\ge c$ minimal with the property that $p$ is odd, and let
$q\ge \frac{c}{2}+1$ minimal with the property $q \equiv 2\;\; (\modu 4)$
-- so $q\ge 6$. We will define an embedding of a graph
$G$ containing $K_{p,2(q-1)}$ as a spanning subgraph, so that the dual is simple and has a $1$-cut.
As $K_{p,2(q-1)}$ is a spanning subgraph with $c'=\min\{p,2(q-1)\}$, we have that $G$ is
$c'$-connected and therefore also $c$-connected.

%\begin{figure}[h!t]
%	\centering
%	\includegraphics[width=\textwidth]{odd_case.pdf}
%	\caption{An embedding of $K_{p,q}$ with $p$ odd and $q\equiv2 (\modu 4)$.}
%	\label{fig:oddcase}
%\end{figure}

\begin{figure}[h!t]
	\centering
	\includegraphics[width=0.8\textwidth]{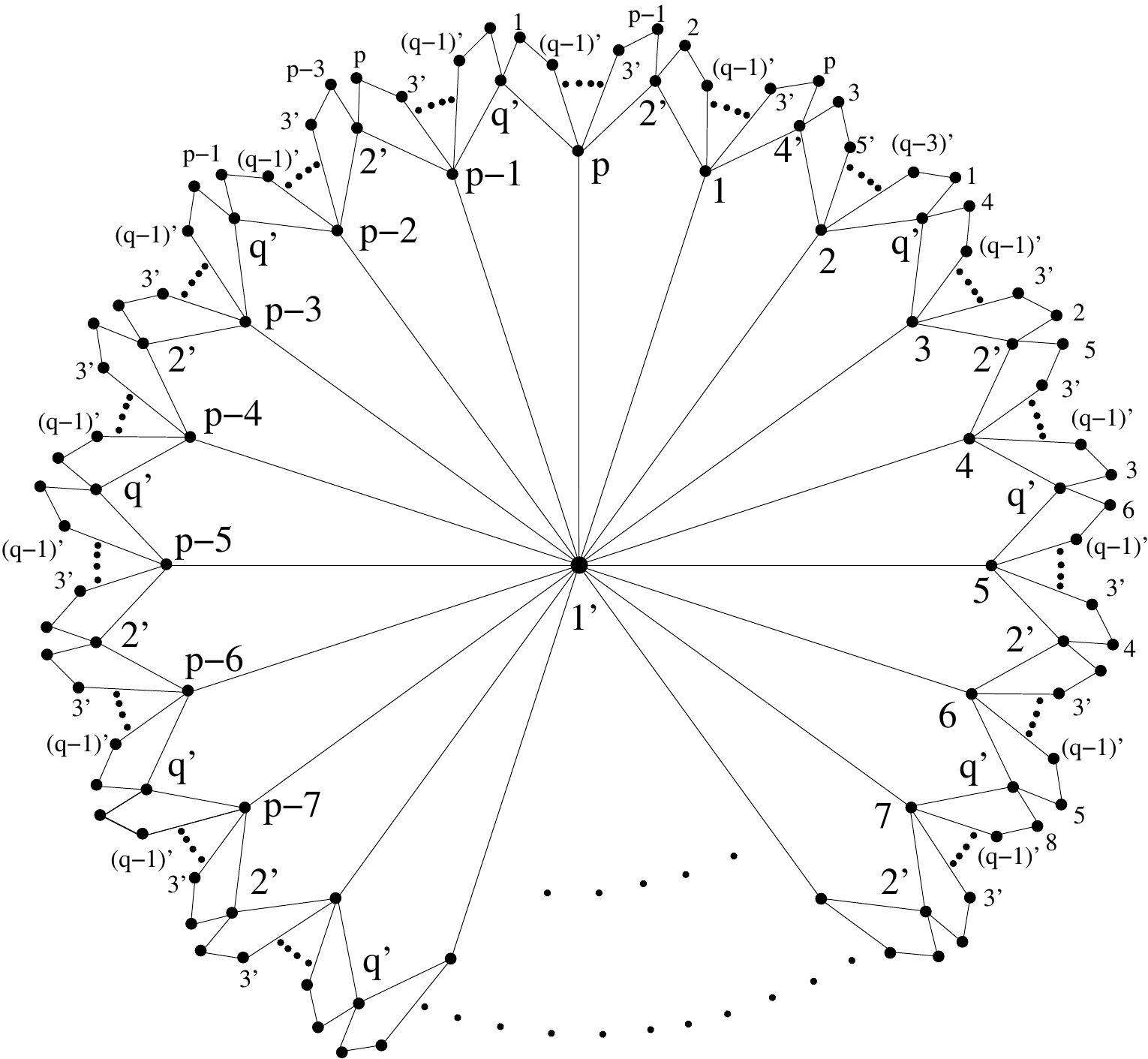}
	\caption{An embedding of $K_{p,q}$ with $p$ odd and $q\equiv 2 \;\; (\modu 4)$.}
	\label{fig:oddcase_1}
\end{figure}

In Figure~\ref{fig:oddcase_1}, a part of a minimum genus embedding of $K_{p,q}$ as
described by Ringel \cite{Ringel1965} is displayed. If the vertex bipartition is
$V_p=\{1,2,\dots ,p\},V_q=\{1',2',\dots ,q'\}$,
then the cyclic order around the vertices
given by Ringel is

%\pagebreak[1]
\begin{tabular}{ll}
For $1\in V_p$: & $(q-1)',q',(q-3)',(q-2)',\dots ,3',4',1',2'$ \\
 & (alternating index differences $+1$ and $-3$).\\
For $2\in V_p$: & $2',3', 6',7',\dots ,q',1',\dots  ,(q-5)',(q-2)',(q-1)'$\\
 & (alternating index differences $+1$ and $+3$ \\
 & and replacing $(q+1)'$ by $1'$). \\
For odd $i\in V_p, i\ge 3$: & $q', (q-1)', \dots ,1'$.\\
For even $i\in V_p, i\ge 3$: & $1', 2', \dots ,q'$.\\
For odd $i'\in V_q$: & $1, 2, \dots ,p$.\\
For even $i'\in V_q$: &  $p, p-1, \dots ,1$.\\
\end{tabular}

For $p$ odd and $q\equiv 2\;\; (\modu 4)$, the genus is equal to $\frac{(p-2)(q-2)}{4}$
and all faces are quadrangles \cite{Ringel1965}, so (see the proof of Lemma~\ref{lem:completedualsimple})
no two faces can share more than one edge
and the dual is simple.

Removing vertex $1'\in V_q$,
we get one big face with all vertices of $V_p$ in the boundary. In this new embedded graph $G_1$, some
of the old quadrangles share 2 edges with the new, large face. The pattern in which faces occur two
times in the boundary can be described based on the rotation system, but can best be seen
in Figure~\ref{fig:oddcase_2}. In order to make sure that each face shares only one
edge with another face, we construct the graph $G_2$ by adding
edges $\{4k-1,4k\}$ and $\{4k,4k+1\}$,
for $1\le k\le \frac{p}{4}$, and
an additional edge $\{p,1\}$, if $p\equiv 3 \;\; (\modu 4)$.

\begin{figure}[h!t]
	\centering
	\includegraphics[width=0.75\textwidth]{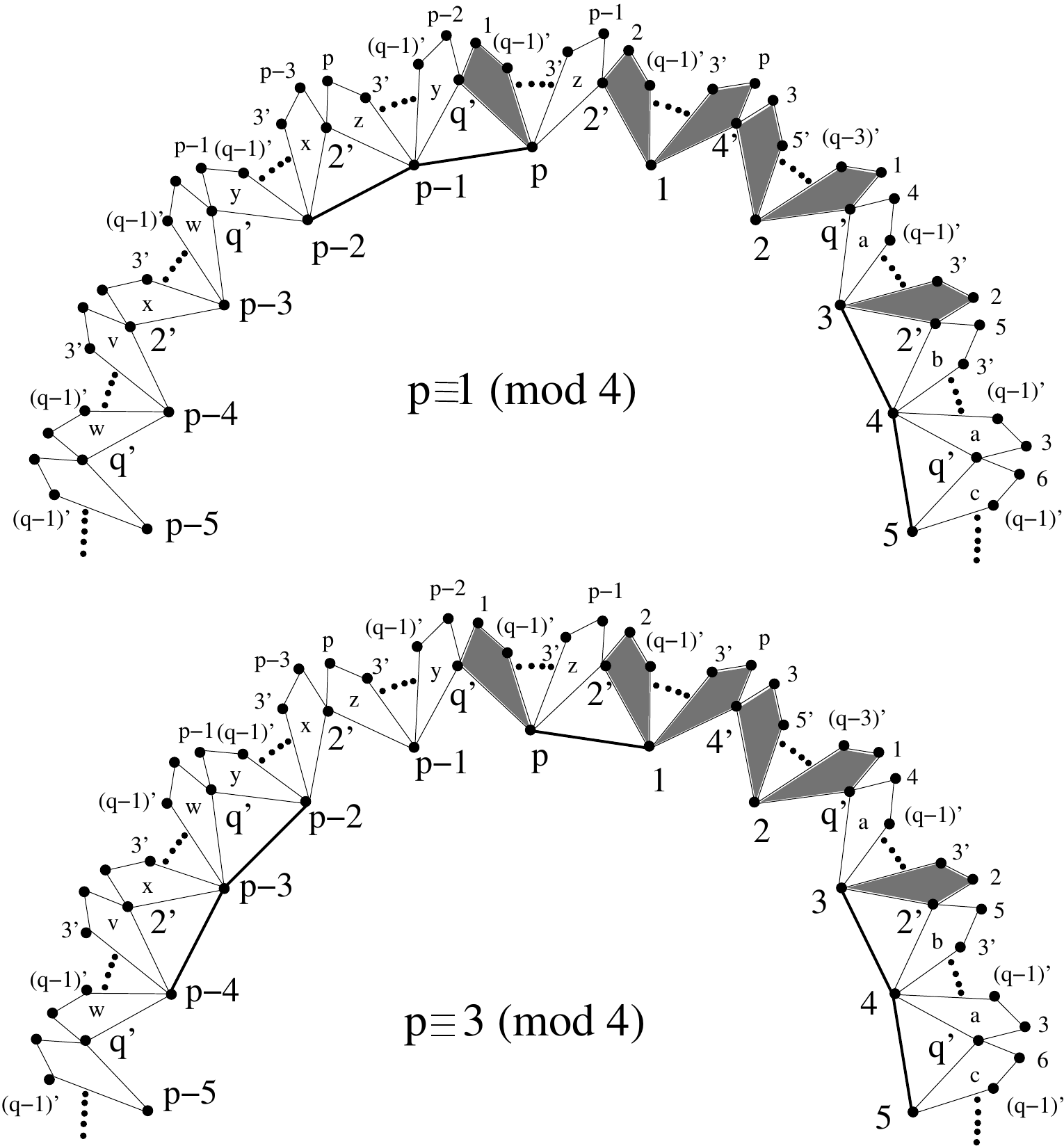}
	\caption{Quadrangles sharing just one edge with the large face are shaded. The other quadrangles are assigned letters
to indicate the two places where they occur in the boundary of the large face. The graph $G_2$ is formed by
adding new edges to avoid that a quadrangle shares more than one edge with the large face.}
	\label{fig:oddcase_2}
\end{figure}

$G_1$ is isomorphic to $K_{p,(q-1)}$. Taking two copies $G^a_{1},G^b_1$
of $G_1$ and identifying the vertex sets
$V^a_p, V^b_p$ with $p$ vertices in any way using a bijection, we get $K_{p,2(q-1)}$.
Doing the same with copies $G^a_2=(V^a_p\cup V^a_q,E^a)$ and
$G^b_2=(V^b_p\cup V^b_q,E^b)$ of $G_2$ we get a graph that contains
$K_{p,2(q-1)}$ as a spanning subgraph, so it is at least $c$-connected. Since in $G^a_2$ and $G^b_2$, some vertices in $V^a_p$ respectively $V^b_p$ are
adjacent, and because we neither
want to delete edges, nor create double edges, we will have to identify $V^a_p$ and $V^b_p$
in a way that no two vertices in $V^a_p$
that are adjacent in $G^a_2$ are identified with vertices adjacent in $G^b_2$.

\begin{figure}[h!t]
	\centering
	\includegraphics[width=0.5\textwidth]{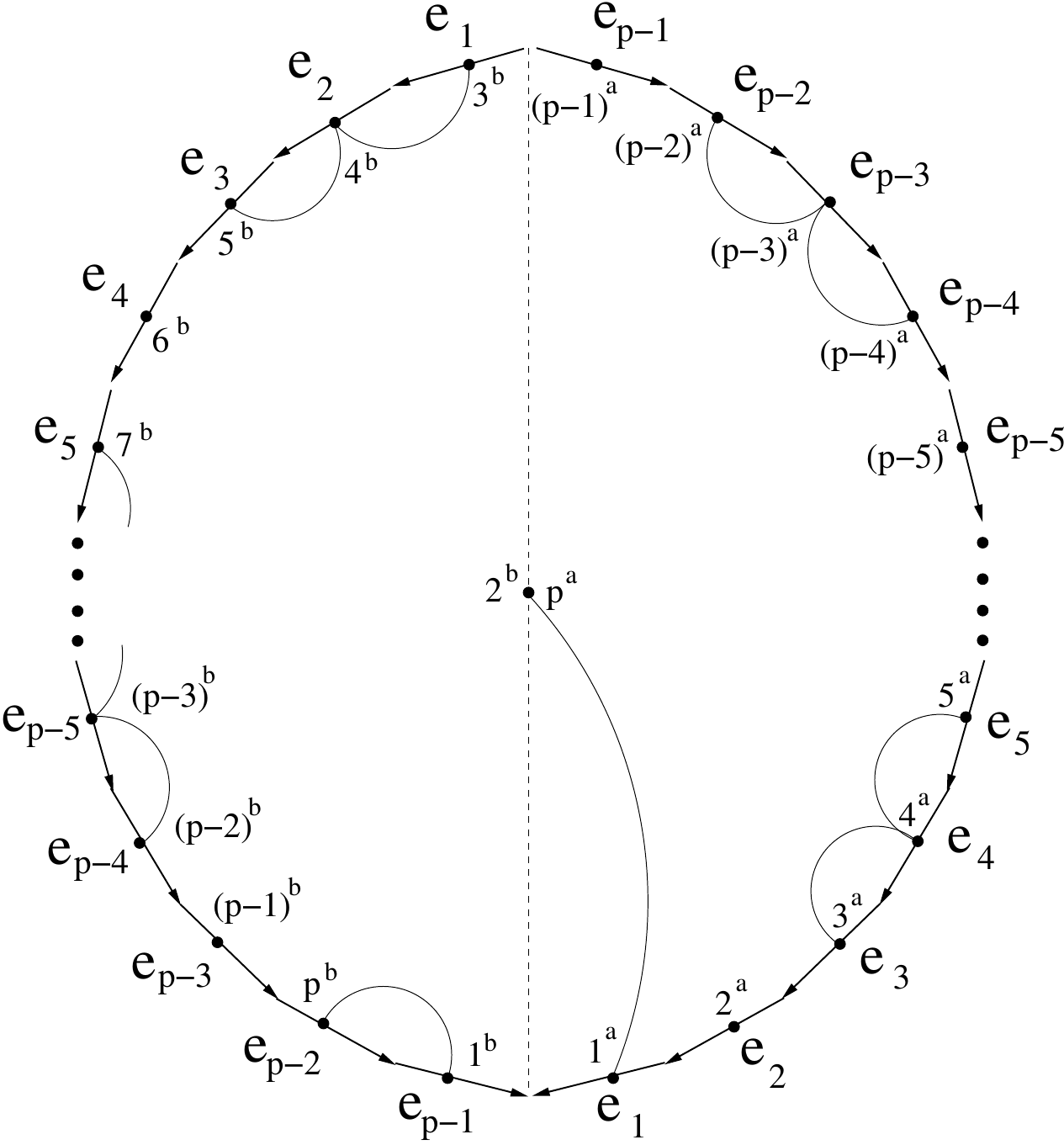}
	\caption{A fundamental polygon split into two parts and with instructions on how $G^a_2$ and $G^b_2$ are
embedded into the two parts. The identification along the boundary is described
by labelling arrows to be identified by the same symbol $e_i$. The gluing into the polygon is described by giving the
positions of the vertices $i^a$, respectively $i^b$. The extra edges not belonging to $K_{p,q-1}$ are given.}
	\label{fig:oddcase_fundpol}
\end{figure}

Denoting the vertices in $V^b_p$ as $1^b, 2^b, \dots$
in order to distinguish them from vertices in $V^a_p$,
which we denote as $1^a, 2^a, \dots$,
we identify, for $1\le i\le p-2$, vertex $i^a$ with vertex $(i+2)^b$, vertex $(p-1)^a$ with
$1^b$ and vertex $p^a$ with $2^b$. The rotation  around the vertices is given by
adding the edges coming from the other graph in the formerly large face obtained by removing vertex $1' \in V_q$.
This identification is displayed in Figure~\ref{fig:oddcase_fundpol}, where, for the case
$p\equiv 3 \;\; (\modu 4)$, also the edges between vertices
of $V^a_p$ and the edges between vertices of $V^b_p$
are drawn to show that no double edges exist. It is easy to check that this also
holds for $p\equiv 1 \;\; (\modu 4)$.

With $s(e)$ denoting the starting point of an arrow and $t(e)$ denoting the endpoint, for $1\le i \le p-2$, the right hand side of the
fundamental polygon gives $s(e_{i})=t(e_{i+1})$ and the left hand side gives $t(e_{i+1})=s(e_{i+2})$,
for $0\le i \le p-3$.
Together this gives $s(e_{i})=s(e_{i+2})$, for $1\le i \le p-3$. This means that all starting points of arrows with odd
index are the same and all starting points of arrows with even index are the same. Together with
$s(e_1)=s(e_{p-1})$ (note that $p-1$ is even) and $t(e_{p-1})=t(e_1)=s(e_2)$,
this gives that all start- and endpoints of arrows in the fundamental polygon correspond to the same point.
We get exactly one face that is not one of the triangles or quadrangles contained in $G^a_2$ and $G^b_2$.

The same conclusion can also be obtained without use of the fundamental polygon in
Figure~\ref{fig:oddcase_fundpol} and arguing only with the rotation  around the vertices.

In order to compute the genus of the graph, we can neglect the edges added after removing the vertex $1'\in V_q$
and compute the genus of the graph without these edges: each of the edges subdivides a face, so we have one
more edge and one more face, and the Euler characteristic does not change.

The embedding of $K_{p,q}$ is a minimum genus embedding with all faces quadrangles,
so it has $p+q$ vertices, $pq$ edges and $(pq)/2$ faces. After removing vertex $1'$, the graph has
$p+q-1$ vertices, $pq-p$ edges and $(pq)/2-p+1$ faces.
If $G_3$ is the result of identifying the vertices, $G_3$ has $v(G_3)=2(p+q-1)-p$ vertices,
$e(G_3)=2(pq-p)$ edges and, because during the identification
the two large
faces are replaced by one new face, $f(G_3)=  2((pq)/2-p+1)-2+1=pq-2p+1$ faces.
The genus of the resulting graph $G_3$ equals

%\[ 2g(G_3)=2-2(p+q-1)+p+2(pq-p)-pq+2p-1\]

%\[ g(G_3)=  1-(p+q-1)+p/2+pq-p- \frac{pq}{2} +p-1/2= \frac{(p-2)(q-1)+1}{2} \]

\[ g(G_3)=  \frac{(p-2)(q-1)+1}{2}. \]

Because all edges in $G_3$ have one of the small faces in $G^a_2$ and $G^b_2$ on one side, no
two faces share more than one edge and the dual is a simple graph. All paths between vertices
in the dual corresponding to small faces in different copies must pass through
the new large face --
so the vertex corresponding to the new face is a cut-vertex of the dual graph.

Since we have chosen $p$ and $q$ minimal, we have $c \le p\le c+1$ and $c/2+1\le q \le c/2+\frac{9}{2}$.
As $q>1$, we have

\[\delta_1(c)\le g(G_3) \le \frac{c^2+6c-5}{4}.\]

\end{proof}

\section{The uniqueness of graphs with high connectivity and small genus}

A key to investigate, for which $c$
we have $\delta_2(c)=\left\lceil \frac{(c-2)(c-3)}{12}\right\rceil+1,$
and, for which $c$, we have $\delta_2(c)=\left\lceil \frac{(c-2)(c-3)}{12}\right\rceil$,
is provided by
Plummer and Zha \cite{g-unique}. Their Theorem~2.4~(A) states

\begin{theorem}\label{thm:plumzha}
Suppose $c\ge 7$ and let $g$ be the genus of the complete graph $K_{c+1}$.
Then $K_{c+1}$ is the only $c$-connected graph that has an embedding of
genus $g$ if and only if $c\not\in \{7,8,9,10,12,13,16\}$.
\end{theorem}

In fact in \cite{g-unique}, the uniqueness of
the complete graph for
$c=9$ and $c=13$ is not decided, and
is explicitly posed as an open question.
In the appendix,
we give an embedding of the $9$-connected graph $K_{11}$ minus a maximum matching
with genus $g=g(K_{10})=4$, and an embedding of the $13$-connected graph
$K_{15}$ minus a maximum matching with genus $g=g(K_{14})=10$, showing
that, for these last two cases, the complete graphs are also not unique.
The embeddings were computed by the program described in \cite{genuscomp}.

Theorem~\ref{thm:plumzha} implies that, in order to decide whether
$\delta_2(c)=\left\lceil \frac{(c-2)(c-3)}{12}\right\rceil+1$ or
$\delta_2(c)=\left\lceil \frac{(c-2)(c-3)}{12}\right\rceil$,
it is  -- except for a finite number of exceptions -- sufficient to study only genus embeddings of
complete graphs and decide whether their dual can be a simple graph with a 2-cut.

\section{Conclusions, future work, and further results}

Though the general bounds for higher genus are relevant, it was most
important to solve the problem for the first nontrivial case -- the torus -- completely,
that is,
be able to give exact values for the minimum
connectivities that guarantee $3$-connectivity,
respectively $2$-connectivity
of the dual.

It was also astonishing to see that, if $g(c)$ is the minimum genus
on which a $c$-connected graph can be embedded, already on genus $g(c)+1$ and maybe even on genus $g(c)$,
$c$-connectedness does not guarantee $3$-connectivity of the dual.

The fact that arbitrarily highly connected graphs can even have
a cutvertex in the (simple) dual is also intriguing -- though this may happen only for much higher genus
than the occurrence of $2$-cuts.

Nevertheless there are still many relevant open questions:

\begin{itemize}

\item It would be very interesting to characterize when
$\delta_2(c)=\left\lceil \frac{(c-2)(c-3)}{12}\right\rceil$ and when $\delta_2(c)=\left\lceil \frac{(c-2)(c-3)}{12}\right\rceil+1$.

\item The upper bounds for $\delta_1()$ are very far from the lower bounds.
Using the same techniques as in the proof of the upper bound, a small improvement might be possible
by choosing $p,q$ less generous and also considering the cases for
bipartite graphs when $q \not\equiv 2\;\; (\modu 4)$. For a substantial
improvement of the upper bound or the lower bound, new ideas are necessary.

%\item In this article we focus on orientable surfaces. Analogous results should be possible
%for non-orientable surfaces.

\item In all examples constructed in this article, the embedded graph with high connectivity
can also be embedded with smaller genus -- so it is not minimum genus embedded.
In Figure~\ref{fig:genusembedding_torus_3_2},
we give an example of a minimum genus embedding of a 3-connected graph on the torus
where the dual has a $2$-cut
and is also minimum genus embedded.
So also minimum genus embeddings of graphs with connectivity at least $3$
exist that have a simple dual that is not
$3$-connected, but also minimum genus embedded.
It would be interesting to know which of the results given
are also valid for minimum genus embeddings.

\begin{figure}[h!t]
	\centering
	\includegraphics[width=0.6\textwidth]{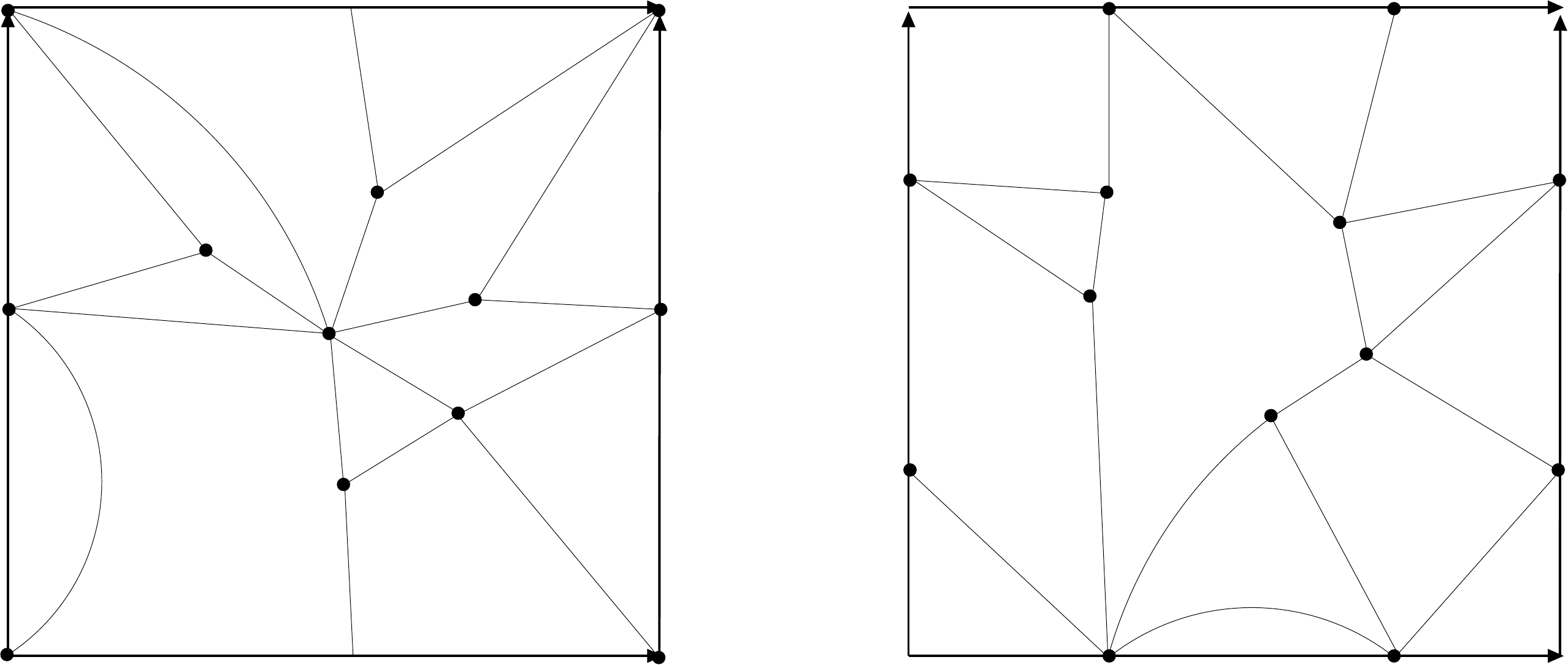}
	\caption{A $3$-connected minimum genus embedded graph on the torus (left)
	with a minimum genus embedded
          dual with a 2-cut (right).}
	\label{fig:genusembedding_torus_3_2}
\end{figure}

\item Due to Whitney's theorem, the statement that a planar 3-connected graph has a planar embedding with a 3-connected dual is
  equivalent to the statement that all its planar embeddings have this property.
For higher genus, the statement that all embeddings have this property is
false, but  does there exist $c\ge 3$ such that, for $c$-connected graphs,
we have that, whenever an embedding with a simple
dual exists, there also exists an embedding with a simple $3$-connected dual of the same genus?

\item In \cite{gocox}, a general approach to local symmetry preserving operations
  (encompassing the dual, truncation, ambo, chamfer, etc.) is described
  and it is proven (Theorem~5.2 in \cite{gocox}) that all operations captured by this approach preserve the
  3-connectedness of polyhedra. In the original manuscript we mentioned the task of extending this result
  to polyhedral embeddings and of classifying operations that always preserve 3-connectedness.
  These aims have in the meantime been achieved.
  In \cite{lopsp-polembed}, the theorem from \cite{gocox} is generalized to polyhedral embeddings and to operations that are only guaranteed to
  preserve orientation preserving symmetries.
  In \cite{Heidi_masterthesis}, a classification of operations
  that always preserve 3-connectedness is given. A publication of this result is in preparation.

%\item Double edges and loops make the results trivial as they allow 1-faces and 2-faces. Allowing
%double edges and loops, but still forbidding 1-faces and 2-faces could make some
%constructions easier and still relevant.

\end{itemize}

\section{Acknowledgement}

Zamfirescu's research was supported by a Postdoctoral Fellowship of the Research Foundation Flanders (FWO).

Bokal's research was supported by Slovenian Research Agency programme P1--0297 and projects J1--8130 and J1--2452.

We want to thank one of the referees for very careful reading and numerous remarks that helped to improve the paper.

%\bibliographystyle{plain}
%\bibliography{/home/gbrinkma/schreib/literatur}\section{Appendix}

\subsection{Embeddings of complete graphs and duals to which the H-operation can be applied}
{\small
\begin{minipage}[t]{4.5cm}
An embedding of $K_8$ \hfill \\
with genus $2$:\hfill \\

\medskip

\begin{tabular} {l|l}
vertex & order of \\
 & neighbours \\
\hline
1 &  2  4  8  3  6  7  5 \\
2 &  1  5  8  7  3  6  4 \\
3 &  1  8  2  7  5  4  6 \\
4 &  1  2  6  3  5  7  8 \\
5 &  2  1  7  4  3  6  8 \\
6 &  5  7  1  3  4  2  8 \\
7 &  6  3  2  8  4  5  1 \\
8 &  7  2  5  6  3  1  4 \\
\end{tabular}
\end{minipage}\hfill
\begin{minipage}[t]{7.5cm}
The dual of the
embedding of $K_8$:\hfill \\

\medskip

\begin{tabular} {l|l||l|l}
vertex & order of & vertex & order of \\
 & neighbours & & neighbours \\
\hline
1 &  2  3  4  & 10 &  4  7  16  \\
2 &  1  5  6  & 11 &  5  16  18  \\
3 &  1  7  8  & 12 &  5  9  17  \\
4 &  1  9  10  & 13 &  6  18  7  \\
5 &  2  11  12  & 14 &  6  9  15  \\
6 &  2  13  14  & 15 &  7  14  18  \\
7 &  3  15  10  13  & 16 &  8  11  10  \\
8 &  3  16  17  & 17 &  8  18  12  \\
9 &  4  12  14  & 18 &  11  13  17  15  \\
\end{tabular}
\end{minipage}

\bigskip

\begin{minipage}[t]{9cm}
The result (genus $3$) of the
H-operation applied to edge $\{1,2\}$ of the dual of the embedding of $K_8$:\hfill \\

\medskip

\begin{tabular} {l|l||l|l}
vertex & order of & vertex & order of \\
 & neighbours & & neighbours \\
\hline
1 &  2  3  4 & 9 &  4  6  12 \\
2 &  1  3  4 & 10 &  4  7  13 \\
3 &  1  5  6  2  7  8 & 11 &  4  14  7 \\
4 &  1  9  10  2  11  12 & 12 &  4  9  16 \\
5 &  3  13  14 & 13 &  8  5  10 \\
6 &  3  9  15 & 14 &  5  11  15  16 \\
7 &  3  16  10  11 & 15 &  8  14  6 \\
8 &  3  13  15 & 16 &  7  12  14  \\
\end{tabular}
\end{minipage}

\bigskip

\begin{minipage}[t]{5.0cm}
An embedding of $K_9$ with genus $3$
and only one face not a triangle:\hfill

\medskip

\begin{tabular} {l|l}
vertex & order of \\
 & neighbours \\
\hline
1 &  2  3  4  5  6  7  8  9 \\
2 &  1  9  6  4  8  5  3  7 \\
3 &  1  6  8  7  2  5  9  4 \\
4 &  1  3  9  8  2  6  7  5 \\
5 &  1  4  7  9  3  2  8  6 \\
6 &  1  5  8  3  4  2  9  7 \\
7 &  1  6  9  5  4  2  3  8 \\
8 &  1  7  3  6  5  2  4  9 \\
9 &  1  8  4  3  5  7  6  2 \\
\end{tabular}
\end{minipage}\hfill
\begin{minipage}[t]{7cm}
The dual of the embedding of $K_9$:\hfill

\medskip

\begin{tabular} {l|l||l|l}
vertex & order of & vertex & order of \\
 & neighbours & & neighbours \\
\hline
1 &  2  3  4  5  6  7 & 13 &  4  22  7 \\
2 &  1  8  9 & 14 &  5  18  20 \\
3 &  1  10  11 & 15 &  6  22  20 \\
4 &  1  12  13 & 16 &  7  21  18 \\
5 &  1  14  9 & 17 &  8  23  11 \\
6 &  1  15  11 & 18 &  8  14  16 \\
7 &  1  16  13 & 19 &  9  12  23 \\
8 &  2  17  18 & 20 &  10  14  15 \\
9 &  2  19  5 & 21 &  10  12  16 \\
10 &  3  20  21 & 22 &  13  15  23 \\
11 &  3  6  17 & 23 &  17  22  19 \\
12 &  4  19  21 \\
\end{tabular}
\end{minipage}

\bigskip

\begin{minipage}[t]{5cm}
An embedding of $K_{10}$ with genus $4$
and only one face not a triangle:\hfill

\medskip

\begin{tabular} {l|l}
vertex & order of \\
 & neighbours \\
\hline
1 &  2  3  4  5  6  7  8  9  10 \\
2 &  1  10  9  7  6  4  8  5  3 \\
3 &  1  2  5  9  6  8  10  7  4 \\
4 &  1  3  7  9  8  2  6  10  5 \\
5 &  1  4  10  7  9  3  2  8  6 \\
6 &  1  5  8  3  9  10  4  2  7 \\
7 &  1  6  2  9  4  3  5  10  8 \\
8 &  1  7  10  3  6  5  2  4  9 \\
9 &  1  8  4  7  2  10  6  3  5 \\
10 &  1  3  8  7  5  4  6  9  2 \\
\end{tabular}
\end{minipage}\hfill
\begin{minipage}[t]{7cm}
The dual of the embedding of $K_{10}$:\hfill

\medskip

\begin{tabular} {l|l||l|l}
vertex & order of & vertex & order of \\
 & neighbours & & neighbours \\
\hline
1 &  2  3  4 & 16 &  7  28  23 \\
2 &  1  5  6 & 17 &  7  25  21 \\
3 &  1  7  8 & 18 &  8  14  25 \\
4 &  1  9  10 & 19 &  9  15  23 \\
5 &  2  11  8  12  9  13 & 20 &  10  26  12 \\
6 &  2  14  15 & 21 &  10  17  29 \\
7 &  3  16  17 & 22 &  11  29  24 \\
8 &  3  5  18 & 23 &  11  16  19 \\
9 &  4  5  19 & 24 &  12  13  22 \\
10 &  4  20  21 & 25 &  13  18  17 \\
11 &  5  22  23 & 26 &  14  20  28 \\
12 &  5  20  24 & 27 &  15  29  28 \\
13 &  5  25  24 & 28 &  16  27  26 \\
14 &  6  26  18 & 29 &  21  27  22 \\
15 &  6  19  27 \\
\end{tabular}
\end{minipage}

\bigskip

\begin{minipage}[t]{6.0cm}
An embedding of $K_{14}$ with genus $10$
and only one face not a triangle:\hfill

\medskip

\begin{tabular} {l|l}
vertex & order of neighbours \\
\hline
1 &  2  3  4  5  6  7  8  9  10  11  12  13  14 \\
2 &  1  14  8  11  9  13  6  12  5  7  4  10  3 \\
3 &  1  2  10  9  5  8  7  14  11  6  13  12  4 \\
4 &  1  3  12  9  6  11  13  8  14  10  2  7  5 \\
5 &  1  4  3  9  14  13  11  7  2  12  8  10  6 \\
6 &  1  5  10  12  2  13  3  11  4  9  8  14  7 \\
7 &  1  6  4  2  5  11  10  13  9  12  14  3  8 \\
8 &  1  7  3  4  13  10  5  12  11  2  14  6  9 \\
9 &  1  8  6  4  12  7  13  2  11  14  5  3  10 \\
10 &  1  9  3  2  4  14  12  6  5  8  13  7  11 \\
11 &  1  10  7  5  13  4  6  3  14  9  2  8  12 \\
12 &  1  11  8  5  2  6  10  14  7  9  4  3  13 \\
13 &  1  12  3  6  2  9  7  10  8  4  11  5  14 \\
14 &  1  13  5  9  11  3  7  12  10  4  6  8  2 \\
\end{tabular}
\end{minipage}\hfill
\begin{minipage}[t]{6.5cm}
The dual of the embedding of $K_{14}$:\hfill

\medskip

\begin{tabular} {l|l||l|l}
vertex & order of & vertex & order of\\
 & neighbours & & neighbours \\
\hline
1 &  2  3  4 & 30 &  16  46  47 \\
2 &  1  5  6 & 31 &  17  48  23 \\
3 &  1  7  8 & 32 &  17  49  26 \\
4 &  1  9  10 & 33 &  19  50  51 \\
5 &  2  11  12 & 34 &  20  46  52 \\
6 &  2  13  14 & 35 &  20  26  53 \\
7 &  3  15  16 & 36 &  21  49  22 \\
8 &  3  17  18 & 37 &  21  40  54 \\
9 &  4  19  20 & 38 &  21  49  55 \\
10 &  4  21  22 & 39 &  22  56  47 \\
11 &  5  23  19 & 40 &  24  37  53 \\
12 &  5  24  25 & 41 &  24  48  29 \\
13 &  6  21  26 & 42 &  25  28  57 \\
14 &  6  27  28 & 43 &  27  56  45 \\
15 &  7  29  21 & 44 &  28  52  51 \\
16 &  7  21  30 & 45 &  29  58  43 \\
17 &  8  31  32 & 46 &  30  55  34 \\
18 &  8  25  21 & 47 &  30  58  39 \\
19 &  9  33  11 & 48 &  31  59  41 \\
20 &  9  34  35 & 49 &  32  36  38 \\
21 &  10  15  36  13  & 50 &  33  57  53 \\
     & 16  37  38  18 & 51 &  33  58  44 \\
22 &  10  39  36 & 52 &  34  59  44 \\
23 &  11  31  27 & 53 &  35  50  40 \\
24 &  12  40  41 & 54 &  37  59  56 \\
25 &  12  18  42 & 55 &  38  46  57 \\
26 &  13  35  32 & 56 &  39  43  54 \\
27 &  14  43  23 & 57 &  42  50  55 \\
28 &  14  42  44 & 58 &  45  51  47 \\
29 &  15  45  41 & 59 &  48  54  52 \\

\end{tabular}
\end{minipage}

}

\subsection{Graphs for the last open cases in the theorem of Plummer and Zha}
{\small
\begin{minipage}[t]{6.5cm}
An embedding  with genus $4$ of the $9$-connected graph
$K_{11}$ minus a maximum matching:\hfill

\medskip

\begin{tabular} {l|l}
vertex & order of neighbours \\

\hline
1 &  3  7  5  11  8  4  10  6  9 \\
2 &  3  6  8  9  11  4  7  10  5 \\
3 &  1  9  8  11  6  2  5  10  7 \\
4 &  1  8  5  9  6  7  2  11  10 \\
5 &  1  7  9  4  8  10  3  2  11 \\
6 &  2  3  11  7  4  9  1  10  8 \\
7 &  6  11  9  5  1  3  10  2  4 \\
8 &  9  2  6  10  5  4  1  11  3 \\
9 &  7  11  2  8  3  1  6  4  5 \\
10 &  8  6  1  4  11  2  7  3  5 \\
11 &  10  4  2  9  7  6  3  8  1  5 \\
\end{tabular}
\end{minipage}\hfill
\begin{minipage}[t]{7.5cm}
An embedding with genus $10$ of the $13$-connected graph
$K_{15}$ minus a maximum matching:\hfill

\medskip

\begin{tabular} {l|l}
vertex & order of neighbours \\
\hline
1 &  3  6  5  10  13  4  15  14  8  11  9  12  7 \\
2 &  3  10  8  9  6  14  7  12  15  4  11  13  5 \\
3 &  1  7  15  10  2  5  9  8  14  12  13  11  6 \\
4 &  1  13  8  10  12  9  5  6  7  14  11  2  15 \\
5 &  1  4  9  3  2  13  15  7  11  8  12  14  10 \\
6 &  2  9  7  4  1  3  11  15  8  13  12  10  14 \\
7 &  6  9  13  10  11  5  15  3  1  12  2  14  4 \\
8 &  9  2  10  4  13  6  15  12  5  11  1  14  3 \\
9 &  7  6  2  8  3  5  4  12  1  11  14  15  13 \\
10 &  8  2  3  15  11  7  13  1  5  14  6  12  4 \\
11 &  10  15  6  3  13  2  4  14  9  1  8  5  7 \\
12 &  13  3  14  5  8  15  2  7  1  9  4  10  6 \\
13 &  11  3  12  6  8  4  1  10  7  9  15  5  2 \\
14 &  12  3  8  1  15  9  11  4  7  2  6  10  5 \\
15 &  14  1  4  2  12  8  6  11  10  3  7  5  13  9 \\
\end{tabular}
\end{minipage}
}

\end{document}